\documentclass[10pt, reqno, a4paper]{amsart}

\usepackage{amsmath,amssymb,amsthm,graphicx,amsxtra, setspace}
\usepackage[utf8]{inputenc}
\usepackage{charter}
\usepackage{mathrsfs}
\usepackage{alltt}
\usepackage{relsize}
\usepackage{hyperref}
\usepackage{aliascnt}
\usepackage{tikz}
\usepackage{mathtools}
\usepackage{multicol}
\usepackage{upgreek}
\usepackage{graphicx,type1cm,eso-pic,color}
\allowdisplaybreaks
\usepackage{scalerel,stackengine}
\stackMath
\newcommand\reallywidehat[1]{%
	\savestack{\tmpbox}{\stretchto{%
			\scaleto{%
				\scalerel*[\widthof{\ensuremath{#1}}]{\kern-.6pt\bigwedge\kern-.6pt}%
				{\rule[-\textheight/2]{1ex}{\textheight}}
			}{\textheight}%
		}{0.5ex}}%
	\stackon[1pt]{#1}{\tmpbox}%
}
\numberwithin{equation}{section}

\newtheorem{theorem}{Theorem}[section]

\newaliascnt{lemma}{theorem}
\newtheorem{lemma}[lemma]{Lemma}
\aliascntresetthe{lemma} 

\newaliascnt{proposition}{theorem}
\newtheorem{proposition}[proposition]{Proposition}
\aliascntresetthe{proposition}

\newaliascnt{assumption}{theorem}
\newtheorem{assumption}[assumption]{Assumption}
\aliascntresetthe{assumption}

\newaliascnt{corollary}{theorem}

\aliascntresetthe{corollary}

\newaliascnt{definition}{theorem}
\newtheorem{definition}[definition]{Definition}
\aliascntresetthe{definition}

\newaliascnt{example}{theorem}

\aliascntresetthe{example}

\newaliascnt{remark}{theorem}
\newtheorem{remark}[remark]{Remark}
\aliascntresetthe{remark}

\newaliascnt{hypothesis}{theorem}

\aliascntresetthe{hypothesis}

\newaliascnt{property}{theorem}

\aliascntresetthe{property}

\let\originalleft\left
\let\originalright\right
\renewcommand{\left}{\mathopen{}\mathclose\bgroup\originalleft}
\renewcommand{\right}{\aftergroup\egroup\originalright}

\makeatletter
\newcommand{\doublewidetilde}[1]{{%
		\mathpalette\double@widetilde{#1}%
}}
\newcommand{\double@widetilde}[2]{%
	\sbox\z@{$\m@th#1\widetilde{#2}$}%
	\ht\z@=.9\ht\z@
	\widetilde{\box\z@}%
}
\makeatother

\renewcommand{\d}{\/\mathrm{d}\/}

\def\w{\textbf{W}^{\varepsilon}_{{\theta}^{\varepsilon}}}

\def\L{\mathrm{L}}
\def\A{\mathrm{A}}
\def\I{\mathrm{I}}
\def\F{\mathrm{F}}
\def\C{\mathrm{C}}

\def\h{\mathbf{h}}
\def\J{\mathrm{J}}
\def\B{\mathcal{B}}
\def\D{\mathrm{D}}

\def\X{\mathbb{X}}
\def\x{\mathbf{x}}

\def\v{\mathbf{v}}
\def\V{\mathbb{V}}
\def\w{\mathbf{w}}

\def\G{\mathbb{G}}

\def\no{\nonumber}
\def\V{\mathbb{V}}

\def\u{\mathbf{u}}
\def\H{\mathbb{H}}
\def\n{\mathbf{n}}

\def\g{\mathbf{g}}

\def\ou{\overline{\mathbf{u}}}

\newcommand{\R}{\mathbb{R}}

\renewcommand{\d}{\/\mathrm{d}\/}

\newcommand{\Addresses}{{
		\footnote{

			\noindent \textsuperscript{1}School of Mathematics,
			Indian Institute of Science Education and Research, Thiruvananthapuram (IISER-TVM),
			Maruthamala PO, Vithura, Thiruvananthapuram, Kerala, 695 551, INDIA.  \par\nopagebreak \noindent
			\textit{e-mail:} \texttt{manikabag19@iisertvm.ac.in}

			\noindent \textsuperscript{2}Department of Mathematical Analysis, Faculty of Mathematics and Physics, Charles University, Sokolovsk\'a 83, 18675 Praha 8, Czech Republic \par\nopagebreak
			\noindent  \textit{e-mail:} \texttt{tanibsw91@gmail.com}
			
			\noindent \textsuperscript{3}School of Mathematics, Indian Institute of Science Education and Research, Thiruvananthapuram (IISER-TVM),
			Maruthamala PO, Vithura, Thiruvananthapuram, Kerala, 695 551, INDIA  \par\nopagebreak \noindent
			\textit{e-mail:} \texttt{sheetal@iisertvm.ac.in}

			\noindent \textsuperscript{*}Corresponding author.

			\medskip\noindent
			{\bf Acknowledgments:}  Manika Bag would like to thank the  Indian Institute of Science Education and Research, Thiruvananthapuram, for providing financial support and stimulating environment for the research. The work of Tania Biswas was supported by the National Research Foundation of Korea(NRF) grant funded by the Korean government(MSIT)(grant No. 2020R1A4A1018190, 2022R1A4A1032094). The work of Sheetal Dharmatti is partially supported by the Government of India SERB  research grant No.  CRG/2021/008278, and the financial assistance is duly acknowledged. 
			
}}}

\begin{document}

	\title[ Cahn-Hilliard-Navier-Stokes Equations with Nonhomogeneous Boundary ]{On the Cahn-Hilliard-Navier-Stokes Equations with Nonhomogeneous Boundary \Addresses	}

	\author[Manika Bag, T. Biswas and S. Dharmatti]
	{Manika Bag\textsuperscript{1}, Tania Biswas\textsuperscript{2} and Sheetal Dharmatti\textsuperscript{3*}}


	\begin{abstract}
		 The evolution of two isothermal, incompressible, immiscible fluids in a bounded domain is governed by Cahn-Hilliard-Navier-Stokes (CHNS) equations. In this work, we study the well-posedness results for the CHNS system with nonhomogeneous boundary condition for the velocity equation. We obtain the existence of global weak solutions in the two-dimensional bounded domain.  We further prove the continuous dependence of the solution on initial conditions and boundary data that will provide the uniqueness of the weak solution. The existence of strong solutions is also established in this work.
	\end{abstract}

 \maketitle
	
	

\section{Introduction}
We consider the motion of an isothermal mixture of two immiscible and incompressible fluids subject to phase separation, which is described by the well-known diffuse interface model. It consists of the Navier-Stokes equations for the average velocity and a convective Cahn-Hilliard equation for the relative concentration, also known as Cahn-Hilliard-Navier-Stokes (CHNS) system or ``model H". A general model for such a system is given by:
\begin{equation}\label{gen}
\left\{
\begin{aligned}
  \partial_t\varphi + \mathbf{u} \cdot \nabla \varphi &= \text{div}(m(\varphi)\nabla\mu), \, \, \text{ in } \Omega \times (0,T), \\
        \mu &= -\Delta\varphi + F'(\varphi), \\
        \partial_t\mathbf{u} - \text{div}(\nu(\varphi) \mathrm{D}\mathbf{u}) + (\mathbf{u}\cdot \nabla)\mathbf{u} + \nabla \pi &= \mu \nabla \varphi, \, \, \text{ in } \Omega \times (0,T), \\
        \text{div}~\mathbf{u} & = 0, \, \, \text{ in } \Omega \times (0,T), \\
\end{aligned}   
\right.
\end{equation}
where $\u(\x, t)$ is the average velocity of the fluid and $\varphi(\x, t)$ is the relative concentration of the fluid. Here, $\Omega$ is a bounded domain in $\mathbb{R}^2, $ with a sufficiently smooth boundary $\Gamma$. The density is taken as matched density, i.e., constant density, which is equal to 1. Moreover, $m$ is mobility of binary mixture, $\mu$ is a chemical potential, $\pi$ is the pressure, $\nu$ is the viscosity and $F$ is a double well potential. The symmetric part of the gradient of the flow velocity vector is denoted by $\mathrm{D}\u$,  that is, $\mathrm{D}\u$ is the strain tensor $\frac{1}{2}\left(\nabla\u+(\nabla\u)^{\top}\right)$. Furthermore, $\mu$   is the first variation of the Helmholtz free energy functional
\begin{align}
\mathcal{E}_1(\varphi) := \int_{\Omega} \left( \frac{1}{2} | \nabla \varphi|^2 + \F(\varphi (x))\right)\, \d x,
\end{align}
where $\F$ is a double-well potential of the regular type. A typical example of regular $\F$ is
\begin{align}\label{regular}
    \F(s)=(s^2-1)^2, \ s \in \mathbb{R}.
\end{align}
A physically relevant but singular potential $\F$ is the Flory-Huggins potential given by
\begin{align*}
\F (\varphi) = \frac{\theta}{2} ((1+ \varphi) \ln (1+\varphi) + (1-\varphi) \ln (1-\varphi)) - \frac{\theta_c}{2} \varphi^2 ,\quad \varphi \in (-1,1),
\end{align*} 
where $\theta,\theta_c>0$.

The model H was derived in \cite{GPV, hh, star} for matched densities.  Whereas, for binary fluids with different densities, more generalized diffuse interface models were proposed in the literature (see, for instance, \cite{abels1, af, agg, boyer1, dss, TL}). There are considerable amount of works devoted to the mathematical analysis of model \eqref{gen} subject to boundary condition 
\begin{align}\label{hbdry}
\frac{\partial\varphi}{\partial\n} = 0, \quad  \frac{\partial\mu}{\partial\mathbf{n}} & = 0, \quad    \u = 0, \quad \text{ on } \Gamma\times(0,T).
\end{align}
 Notably, in \cite{boyer}, the case of $\Omega\subset\mathbb{R}^d$ being a periodical channel and $\F$ being a suitable smooth double-well potential was investigated, with further insights provided in \cite{cg}. A more comprehensive mathematical theory concerning the existence, uniqueness, and regularity of solutions for the system \eqref{gen} with \eqref{hbdry} was developed in \cite{abels}. Generalizations of the model \eqref{gen} have been discussed in \cite{abels1, ADG, ADT}. While \cite{abels} primarily focused on the case of singular potential, all these results hold for regular potentials as well (see \cite{boyer, cg, GG}).
For numerical investigations of this model, references such as \cite{feng, ksw, ls} provide valuable insights. Moreover, a nonlocal version of the model \eqref{gen} was introduced in \cite{gl1, gl2}, wherein the chemical potential $\mu$ is replaced by $a \varphi - \J\ast \varphi + \F'(\varphi)$. Mathematical analysis of this nonlocal model has been conducted in \cite{weak, fri, fri1, strong}, among others. In these studies, the boundary conditions for the velocity field $\u$ have typically been assumed as a no-slip or periodic, while the boundary conditions for the phase field variable $\varphi$ and the chemical potential $\mu$ are often considered as no-flux. The long-time behavior of the system \eqref{gen} with boundary conditions \eqref{hbdry} has been investigated in \cite{abels}, while the existence of global attractors and exponential attractors was explored in \cite{GG}. 

The boundary conditions taken into consideration so far for the CHNS system are rather standard. More precisely, in almost all the contributions, velocity $\u$ is subject to no-slip or periodic boundary conditions, while $\varphi$ and $\mu$ are subject to homogeneous Neumann boundary conditions or periodic boundary conditions in the shear case. However, the incompressible Navier-Stokes equations that describe the motion of a single-phase fluid are well-studied in the literature with homogeneous as well as nonhomogeneous boundary data. We refer to \cite{farwig, fursi, fursi1, raymond} for the treatment of Navier Stokes' equations with nonhomogeneous boundary conditions and references therein.  For Cahn-Hilliard equations, dynamic boundary conditions, which are physically relevant, are analyzed in \cite{ cgs1, gk, mz, cw}. Thus, it would be natural to consider a nonhomogeneous boundary condition for a coupled CHNS system under consideration. For the coupled CHNS system, a dynamic boundary condition for the Cahn-Hilliard equations has been considered in \cite{cher, gk, you}, while \cite{cher} has considered the case of a mixture of compressible fluids. Another more generalized boundary condition (GNBC), which accounts for a moving contact line, is studied in \cite{GGM, GGW, ggp} which consists of the Navier boundary condition for the velocity equations $\u$, a no-flux boundary condition for $\mu$ and a dynamic boundary condition for $\varphi$.


To our knowledge, the well-posedness of the model \eqref{gen} with a nonhomogeneous boundary condition for velocity has not been studied analytically in the literature. If one wants to study a boundary control problem for model H,  well-posedness results are necessary.  However, few numerical results are present in this direction; we refer \cite{ghk, hunt, hunt1}. In these papers, the authors have proved a boundary optimal control problem for a time-discrete CHNS system with nonhomogeneous boundary conditions for the velocity. 

In this article, we want to extend the work of \cite{abels, boyer} in the non-autonomous case by taking a time-dependent boundary condition for velocity. In particular, we consider the system \eqref{gen} with constant mobility, set equal to 1, non-constant viscosity, and nonhomogeneous boundary condition, which is given by
\begin{equation}\label{equ P}
\left\{
\begin{aligned}
 \partial_t\varphi + \mathbf{u} \cdot \nabla \varphi &= \Delta \mu, \, \, \text{ in } \Omega \times (0,T), \\
       \mu &= -\Delta\varphi + F'(\varphi), \\
        \mathbf{u}_t - \text{div}(\nu(\varphi) \mathrm{D}\mathbf{u}) + (\mathbf{u}\cdot \nabla)\mathbf{u} + \nabla \pi &= \mu \nabla \varphi, \, \, \text{ in } \Omega \times (0,T), \\
       \text{div}~\mathbf{u} & = 0, \, \, \text{ in } \Omega \times (0,T), \\
        \frac{\partial\varphi}{\partial\n} = 0, \,  \frac{\partial\mu}{\partial\mathbf{n}} & = 0, \,\, \text{ on } \Gamma\times(0,T), \\
        \u & = \h, \,\, \text{ on } \Gamma\times(0,T), \\
       \mathbf{u}(0) = \mathbf{u}_0 ,\,\, \varphi(0) & = \varphi_0, \,\, \text{ in } \Omega,  
\end{aligned}   
\right.
\end{equation}
where $\h$ is the external force acting on the boundary of the domain.

As a contribution from this work, we have established the existence of weak and strong solutions and the uniqueness of weak solutions for the system  \eqref{equ P}.
Before concluding this section, we want to highlight some key aspects of the paper. The system \eqref{equ P} is inherently non-autonomous owing to the presence of the time-dependent boundary condition $\h(t)$. This introduces some difficulty such as 
 the intricate nonlinear coupling prohibits the direct application of fixed-point arguments. In the context of single-phase fluid flow, such as Navier-Stokes equations with non-homogeneous boundaries, the use of a lifting operator, which maps the boundary data to the domain, has been instrumental in proving the existence of solutions, as demonstrated in \cite{farwig, fursi1, raymond}. Using similar techniques,  we introduce suitable lifting functions (cf. \eqref{equ1} below) for the Navier-Stokes equations. Subsequently, using the lifting defined in \eqref{equ1}, we rewrite the velocity equations with homogeneous boundary conditions and discretize it along with the variable $\varphi$. Then, we study the existence results for the approximate (discretized) problem employing Schauder's fixed-point theorem. Finally, by deriving the necessary estimates, we proceed to the limit of the approximate problem, thereby establishing the existence of a global weak solution. The uniqueness of weak solutions to the CHNS system with unmatched viscosities has previously been established in \cite{GMT}. However, in our setting, the velocity field is subject to nonhomogeneous boundary conditions, which prevents direct application of the techniques developed in \cite{GMT}. Specifically, the difference in concentrations, denoted by $\delta\varphi = \varphi_1 - \varphi_2$ (see Section 4), does not have zero mean due to the influence of boundary data. To address this challenge, we introduce a modification of the strategy proposed in \cite{GMT} to accommodate the non-homogeneous boundary conditions and successfully establish the uniqueness of weak solutions in our framework. Subsequently, we employ a standard regularity argument to derive the existence of strong solutions to the system \eqref{equ P}.

The structure of the paper is as follows:
In Section 2, we present the requisite functional framework to establish the well-posedness results.
Section 3 is dedicated to the establishment of the energy inequality and the proof of existence results utilizing a semi-Galerkin approximation.
Section 4 focuses on demonstrating the continuous dependence of weak solutions on both initial data and boundary terms, thereby establishing the uniqueness of the weak solution.
Moving forward to Section 5, we establish the existence of a strong solution.
\section{Preliminaries}
\subsection{Functional Setup}
Let $\Omega$ be a bounded subset of $\mathbb{R}^2$ with sufficiently smooth boundary $\Gamma$. We introduce the functional spaces that will be useful in the paper. 
\begin{align*}
&\G_{\text{div}} := \Big\{ \u \in \mathrm{L}^2(\Omega;\mathbb{R}^2) : \text{ div }\u=0,\  \u|_{\Gamma}\cdot \mathbf{n}=0 \Big\}, \\
&\V_{\text{div}} :=\Big\{\u \in \mathrm{H}^1_0(\Omega;\mathbb{R}^2): \text{ div }\u=0 \text{ in }\Omega, {\mathrm{H}^{-\frac{1}{2}}(\Gamma)}\langle \u\cdot \mathbf{n}, 1  \rangle_{\mathrm{H}^\frac{1}{2}(\Gamma)}  = \; 0 \Big\},\\ 
&\H^s_{\text{div}} :=\Big\{\u \in \mathrm{H}^s(\Omega;\mathbb{R}^2): \text{ div }\u=0\Big\}, \quad s\geq 0,\\ 
&\V^s(\Omega):= \big\{ \u\in\mathrm{H}^s(\Omega;\mathbb{R}^2): \text{ div }\u=0 \text{ in }\Omega, \; \;  \langle \u\cdot \mathbf{n}, 1  \rangle_{\mathrm{H}^\frac{1}{2}(\Gamma)}  = \; 0 \big\}, \quad s\geq 0,\\
&\mathrm{L}^2:=\mathrm{L}^2(\Omega;\mathbb{R}),\quad 
\mathrm{H}^s:=\mathrm{H}^s(\Omega;\mathbb{R}), \quad s > 0.
\end{align*}
In addition, we define boundary spaces $\mathbb{H}^s(\Gamma; \R^2)$ in the usual trace sense and
\begin{align*}
  \V^s(\Gamma) := \Big\{\h\in \mathrm{H}^s(\Gamma; \R^2) : \int_{\Gamma}\h\cdot\n = 0 \Big\}, \quad s\geq 0. 
\end{align*}
 The dual space of $\mathrm{H}^s(\Omega), \V^s(\Gamma)$ is denoted by $\mathrm{H}^{-s}(\Omega), \V^{-s}(\Gamma)$, respectively.  Let us denote $\| \cdot \|$ and $(\cdot, \cdot)$ the norm and the scalar product, respectively, on  $\V^0(\Omega)$ and $\G_{\text{div}}$. The duality between any Hilbert space $\X$ and its dual $\X'$ will be denoted by $\left<\cdot,\cdot\right>$. We endowed $\V_{\text{div}}$ with the scalar product 
$$(\u,\v)_{\V_{\text{div}} }= (\nabla \u, \nabla \v)=2(\mathrm{D}\u,\mathrm{D}\v),\ \text{ for all }\ \u,\v\in\V_{\text{div}}.$$ The norm on $\V_{\text{div}}$ is given by $\|\u\|_{\V_{\text{div}}}^2:=\int_{\Omega}|\nabla\u(x)|^2\d x=\|\nabla\u\|^2$. Since $\Omega$ is bounded, the embedding of $\mathbb{V}_{\text{div}}\subset\G_{\text{div}}\equiv\G_{\text{div}}'\subset\V_{\text{div}}'$ is compact.

\subsection{\bf Linear and Nonlinear Operators}
Let us define the Stokes operator $\A : \V_{\text{div}} \to \V'_{\text{div}}$ such that 
\begin{equation*}
\langle\A\u, \v\rangle = (\u, \v)_{\V_\text{div}} = (\nabla\u, \nabla\v),  \text{ for all } \u , \v \in \V_{\text{div}}.
\end{equation*}
$\A$ is a canonical isomorphism from $\V_{\text{div}}$ to $\V'_{\text{div}}$. We denote $\A^{-1}:\V'_{\text{div}}\to\V_{\text{div}}$ inverse map of Stokes operator. Then following \cite[Appendix B]{GMT}, 
\begin{align}\label{norm}
   \|\g\|_{\sharp}=\|\nabla\A^{-1}\g\|=\langle\g,\A^{-1}\g\rangle^\frac{1}{2} 
\end{align} 
is an equivalent norm on $\V'_{\text{div}}.$ It should also be noted that  $\A^{-1} : \G_{\text{div}} \to \G_{\text{div}}$ is a self-adjoint compact operator on $\G_{\text{div}}$ and by
the classical \emph{spectral theorem}, there exists a sequence $\lambda_j$ with $0<\lambda_1\leq \lambda_2\leq \lambda_j\leq\cdots\to+\infty$
and a family  $\mathbf{e}_j \in \D(\A)$ of eigenvectors is orthonormal in $\G_\text{div}$ and is such that $\A\mathbf{e}_j =\lambda_j\mathbf{e}_j$. We know that $\u$ can be expressed as $\u=\sum\limits_{j=1}^{\infty}\langle \u,\mathbf{e}_j\rangle \mathbf{e}_j,$ so that $\A\u=\sum\limits_{j=1}^{\infty}\lambda_j\langle \u,\mathbf{e}_j\rangle \mathbf{e}_j$. Thus, it is immediate that 
\begin{equation}
\|\nabla\u\|^2=\langle \A\u,\u\rangle =\sum_{j=1}^{\infty}\lambda_j|\langle \u,\mathbf{e}_j\rangle|^2\geq \lambda_1\sum_{j=1}^{\infty}|\langle \u,\mathbf{e}_j\rangle|^2=\lambda_1\|\u\|^2,
\end{equation} 
which is the \emph{Poincar\'e inequality}.

For later use, we denote $(x\otimes y)_{ij}=x_iy_j, i,j=1,2$, then we can write $$(\mathbf{x}\cdot\nabla)\mathbf{y}=\text{div}(\mathbf{x}\otimes\mathbf{y}).$$

Using the nonlinearity present in the Navier-Stokes equation,
we define some bilinear operators $\B, \B_1, \B_2 \text{ and } \B_3$  as follows:\\
First, let us denote by
$$b(\u,\v,\w) = \int_\Omega (\u(x) \cdot \nabla)\v(x) \cdot \w(x)\d x=\sum_{i,j=1}^2\int_{\Omega}u_i(x)\frac{\partial v_j(x)}{\partial x_i}w_j(x)\d x.$$
Then an integration by parts yields,
\begin{equation*}
\left\{
\begin{aligned}
b(\u,\v,\v) &= 0, \ \text{ for all } \ \u,\v \in\V_\text{{div}},\\
b(\u,\v,\w) &=  -b(\u,\w,\v), \ \text{ for all } \ \u,\v,\w\in \V_\text{{div}}.
\end{aligned}
\right.\end{equation*}
 We define $\B$ : $\V_{\text{div}} \times \V_{\text{div}} $ $\rightarrow$  $\V_{\text{div}}'$ defined by,
$$ \langle \B(\u,\v),\w  \rangle := b(\u,\v,\w), \  \text{ for all } \ \u,\v,\w \in \V_\text{{div}}.$$
$\B_1$ from $\V_{\text{div}} \times \V^1(\Omega)$ into $\V_{\text{div}}'$ defined by,
$$ \langle \B_1(\u,\v),\w  \rangle := b(\u,\v,\w), \  \text{ for all } \u \in \V_\text{{div}} \, \v, \w \in \V^1(\Omega). $$
Similarly we define $\B_2 : \V^1(\Omega)\times \V_{\text{div}} \rightarrow \V_{\text{div}}'$ and $\B_3 : \V^1(\Omega) \times \V^1(\Omega) \rightarrow \V^1(\Omega)' $ defined by
$$ \langle \B_2(\u,\v),\w  \rangle  := b(\u,\v,\w), \  \text{ for all } \u \in \V^1(\Omega) \text{ and }\w,  \v \in \V_\text{{div}}, $$
$$\text{and} \;   \langle \B_3(\u,\v),\w  \rangle := b(\u,\v,\w), \  \text{ for all } \u, \v, \w \in \V^1(\Omega) \; \; \;\text{respectively}. $$

For more details about  linear and nonlinear operators, we refer the readers to \cite{temam}. 

 We recall the following result from \cite[Proposition C.1]{GMT}, which will be utilized later in our calculations.
	\begin{proposition}
	Let $\Omega$ be a bounded domain with a smooth boundary in $\mathbb{R}^2$. Assume that $f, g \in  \mathrm{H}^1$. Then, there exists a positive constant $C$ such that
	\begin{eqnarray}\label{prop2.1}
\|fg\| \leq C \|f\|_{\mathrm{H}^1} \|g\| [\log (e^{\frac{\|g\|_{\mathrm{H}^1}}{\|g\|}})]^{\frac{1}{2}}	 
	\end{eqnarray}
	\end{proposition}
In addition we recall the following lemma from \cite[Lemma 4.2]{temam1}:
\begin{lemma}\label{lem4.2}
    For any $\eta>0$ $$\{\|\Delta\varphi\|^2+\eta\|\varphi\|^2\}^\frac{1}{2}, \quad \text{and} \quad \Big\{\|\Delta\varphi\|^2+\eta\Big(\int_\Omega\varphi(x) dx\Big)^2\Big\}^\frac{1}{2}$$ are norms on $V$, which are equivalent to the $\mathrm{H}^2$ norm, where the space  $V$ is defined by $$V=\{\phi\in\mathrm{H}^2(\Omega):\frac{\partial\phi}{\partial n}=0 \text{ on }\Gamma\}.$$
\end{lemma}
	Let the viscosity $\nu$ and the double-well potential $F$ satisfy the following assumptions:
\begin{assumption}\label{prop of F} 
\begin{enumerate}
 \item[(A1)] We assume the viscosity coefficient $\nu \in W^{1,\infty}(\mathbb{R})$ and for some positive constants $\nu_1$ and $\nu_2$ satisfies $$0 < \nu_1 < \nu(r) < \nu_2.$$ 
 \item[(A2)] $F\in C^3(\mathbb{R}).$ 
\item [(A3)] There exist $C_1 >0$, $C_2 \geq 0$ and $3\leq q\leq 5 $ such that $|F'(s)| \leq C_1|s|^q + C_2,$ for all $s \in \mathbb{R}$. 
 \item [(A4)] There exist $C_3 >0$ such that $F''(s)\geq -C_3$, for all $s \in \mathbb{R}$, a.e., $x \in \Omega$.
	\item [(A5)] There exist $C_4 >0$, $C_4' > 0$  such that $|F''(s)| \leq C_4|s|^{q-1} + C_4'$, for all $s \in \mathbb{R}$, $3 \leq q \leq 5$ and a.e. $x \in \Omega$.
		\item[(A6)] There exist $C_5 > 0$, $|F'''(s)| \leq C_5(1 + |s|^{q-2})$ for all $s \in \mathbb{R} \text{ where } 3\leq q \leq 5$.
		\item[(A7)] $F(\varphi_0) \in \mathrm{L}^1(\Omega)$.
	\end{enumerate}
\end{assumption}
\begin{remark}\label{nu}
From (A1), we conclude that there exists a positive constant $\nu$ such that $\nu(\varphi)-\nu_1\geq\nu>0.$
\end{remark}
Now onward, we consider $C$ as a non-negative generic constant that may depend on initial data, $\nu$, $\nu_1, \, \nu_2$, $F$, $\Omega$, $C_i, \, i=1, \cdots 5$, and the norm of given boundary data $\h$. The value of $C$ may vary even within the proof.
\subsection{Elliptic Lifting Operator}
To ensure the well-posedness and strong solution of the considered system, we introduce a suitable lifting operator. This approach enables a rigorous framework for addressing the existence and uniqueness, which are discussed in detail in the next section. We describe the elliptic lifting operator below.\\
\begin{assumption}\label{H}
Let $\h$ satisfy the following conditions
\begin{equation*}
\left\{
    \begin{aligned}
    \h & \in \mathrm{L}^2(0, +\infty;\V^\frac{3}{2}(\Gamma)) \cap \mathrm{L}^\infty(0, +\infty; \V^\frac{1}{2}(\Gamma)), \\
 \partial_t \h & \in \mathrm{L}^2(0, +\infty; \V^{\frac{1}{2}}(\Gamma)).
 \end{aligned}
 \right.
 \end{equation*}
 \end{assumption}
\begin{theorem} \label{elliptic lifting}
For each $ t \in [0, \infty),$ consider  the stationary Stokes'  equation with non-homogeneous boundary conditions given by 
\begin{equation}\label{equ1}
\left\{
 \begin{aligned}
     -\nu_1 \Delta \u_e(t) + \nabla p(t) & = 0, \,\, \text{ in } \Omega , \\
    \text{div}~\u_e(t) & = 0, \, \, \text{ in } \Omega,\\
    \u_e(t) & = \h(t), \, \, \text{ on } \Gamma,
     \end{aligned}   
     \right.
\end{equation}
Then, under the Assumption \ref{H}, system \eqref{equ1} admits a unique  solution with the regularity 
 \begin{align}\label{e18}
  \u_e \in \mathrm{H}^1(0, +\infty; \mathbb{V}^0(\Omega)) \cap \mathrm{L}^\infty(0, +\infty;\V^1(\Omega)) \cap \mathrm{L}^2(0, T; \V^2(\Omega)),
\end{align}
such that
\begin{align}
    \int_{0}^{t} {\| \u_e(\tau) \|}^2_{\V^2(\Omega)} \, d\tau &\leq c \int_{0}^{t} {\| \h(\tau) \|}^2_{\V^{\frac{3}{2}}(\Gamma)} \, d\tau,\label{estimate h} \\
    \int_{0}^{t} {\| \partial_t\u_e(\tau) \|}^2 \, d\tau &\leq c \int_{0}^{t} \| \partial_t \h(\tau) \|_{\V^{\frac{1}{2}}(\Gamma)} \, d\tau.\label{estimate ht}
    \end{align}
    for all $t \geq 0$ and for some $c$,  a positive constant depending on $\Omega, \nu_1.$
\end{theorem}
\begin{proof}
The inequality \eqref{estimate h} can be proved using standard elliptic regularity theory.  For instance from \cite[Proposition 2.3]{temam} we have 
\begin{align}
    \|\u(t)\|_{\V^1(\Omega)}\leq \,  c \, \|\h(t)\|_{\V^\frac{1}{2}(\Gamma)}, \quad 
    \int_0^t\|\u(\tau)\|^2_{\V^2(\Omega)} \, d\tau\leq \,   c\int_0^t\|\h(\tau)\|^2_{\V^\frac{3}{2}(\Gamma)} \, d\tau,\label{1}
\end{align}
for all $t\geq 0$. Similarly, we differentiate the equation \eqref{equ1} with respect to $t$ and  apply the above inequality for $\partial_t\u_e$. Then, taking the essential supremum on the first inequality of \eqref{1}, we prove the theorem. For details, see \cite[Proposition 2.3]{temam}.
\end{proof}
We call a function $\u_e$ defined as above an elliptic lifting. Using the lifting function introduced above, system \eqref{equ P} can be rewritten
as a system with a homogeneous boundary, which will be used to obtain suitable a priori estimates for the solution via the energy method. We give detailed proof in the next section.

\section{Existence of  a weak solution}

Let us rewrite system \eqref{equ P} 
as a system with a homogeneous boundary using the lifting function defined in the previous section. For, 
set $\overline{\u}= \u-\u_e,$ where $\u_e$ is an elliptic lifting.
Then $(\overline{\u}, \varphi)$ solves the following system: 
\begin{equation}\label{7}
\left\{
 \begin{aligned}
  & \partial_t\varphi +\overline{\u}\cdot \nabla \varphi + \u_e \cdot \nabla \varphi = \Delta\mu, \, \, \text{ in } \Omega \times (0,T),\\
    &\mu = - \Delta\varphi + F'(\varphi),\text{ in } \Omega \times (0,T),\\
   &\partial_t\overline{\u} - \text{div}((\nu(\varphi)-\nu_1) \mathrm{D}\overline{\u}) + ((\overline{\u}+\u_e)\cdot \nabla)(\overline{\u}+\u_e) + \nabla\overline{\pi} = \mu \nabla \varphi - \partial_t\u_e, \, \, \text{ in } \Omega \times (0,T),  \\
    &\text{div}~\overline{\mathbf{u}}  = 0, \, \, \text{ in } \Omega \times (0,T), \\
 &\frac{\partial\varphi}{\partial\n} = 0, \,  \frac{\partial\mu}{\partial\mathbf{n}}  = 0, \,\, \text{ on } \Gamma\times(0,T), \\
  & \overline{\u}  = 0, \,\, \text{ on } \Gamma\times(0,T),  \\
   &     \overline{\u}(0) = \overline{\u}_{0}= \u_0-\u_e(0), \,\, \varphi(0)  = \varphi_0 \,\, \text{ in } \Omega .   
 \end{aligned}
 \right.
 \end{equation}
 
In this section, we prove the existence of a weak solution of \eqref{equ P} by examining the system \eqref{7}. Before going further, let us define the notion of a weak solution of \eqref{equ P}.
\begin{definition}\label{weak sol def}
 Let the function $\u_e$ solve \eqref{equ1}. A pair $(\u, \varphi)$ is said to be a weak solution of the system \eqref{equ P} if $\u = \overline{\u} + \u_e$, and $(\overline{\u},\varphi)$ obeys 
\begin{itemize}
    \item The following weak formulation:
\begin{align}\label{weak formulation u}
    \langle \partial_t\overline{\u}(t),\mathbf{v} \rangle + (\nu(\varphi)-\nu_1)(\nabla \overline{\u}(t),\nabla\mathbf{v}) +  \langle(\overline{\u}(t) + \u_e(t)) \cdot \nabla (\overline{\u}(t) + \u_e(t)),\mathbf{v}\rangle\no\\ = (\mu\nabla \varphi,\mathbf{v}) - \langle \partial_t\u_e,\mathbf{v} \rangle, \quad\text{ holds for } \mathbf{v}\in \V_{\text{div}}, \text{a.e. in } \Omega \times (0, T),
\end{align}
 \begin{align}\label{weak formulation phi}
    \langle \partial_t\varphi(t), \psi \rangle + \big((\overline{\u}(t) + \u_e(t)) \cdot\nabla \varphi(t), \psi \big) + (\nabla \mu(t
    ), \nabla\psi) = 0,
\end{align}
holds for $\psi \in \mathrm{H}^1, \text{ a.e. } t \in (0, T)$. 
\item The initial conditions are satisfied in the weak sense:
    \begin{align}\label{weak formulation initial cond}
    (\overline{\u}(0), \v)  \to (\overline{\u}_0, \v), \quad (\varphi(0), \psi) \to (\varphi_0, \psi) \text{ as } t \to 0, \quad \text{for all }\v\in\V_{\text{div}} \text{ and } \psi\in\mathrm{L}^2. 
\end{align}
\item $(\overline{\u},\varphi)$ satisfies:
    \begin{equation}\label{weak 2}
\left\{
\begin{aligned}
    \overline{\u} & \in \mathrm{L}^\infty(0, T; \G_{\text{div}}) \cap \mathrm{L}^2(0, T; \V_{\text{div}}) \\
   \partial_t \overline{\u} & \in \mathrm{L}^2(0, T; \V_{\text{div}}')\\
    \varphi & \in \mathrm{L}^\infty(0, T; \mathrm{H}^1) \cap \mathrm{L}^2(0, T; \mathrm{H}^2)\\
    \partial_t\varphi & \in \mathrm{L}^2(0, T; ({\mathrm{H}^1})')\\
    \mu & \in \mathrm{L}^2(0, T; \mathrm{H}^1).
\end{aligned}
\right.
\end{equation}
\end{itemize}
\end{definition} 
\begin{remark}
From \eqref{weak 2} it also follows that $\overline{\u} \in C([0,T]; \G_{\text{div}})$ and $\varphi \in C([0,T];\mathrm{L}^2)$.
\end{remark}
Let us first derive the energy estimate satisfied by the lifted system \eqref{7}.\\
\textbf{Energy Estimate:  }
  Let us consider the  energy functional corresponding to the lifted system $\eqref{7}$ given by,
 \begin{align}\label{1energy}
     \mathcal{E}(\overline{\u}(t), \varphi(t)) &:= \frac{1}{2} \|\overline{\u}(t)\|^2 +\frac{1}{2} \|\nabla\varphi(t) \| +\int_\Omega F(\varphi(t)).
\end{align}
Then we can derive the following basic energy inequality for the lifted system $\eqref{7}$:
\begin{lemma}\label{lem 3.2}
    Let $(\u, \varphi)$ be a weak solution of the system \eqref{7} with initial conditions $(\overline{\u}_0, \varphi_0)\in\G_{\text{div}}\times\mathrm{H}^1$. Also, let us assume $\h$ satisfies Assumption \ref{H}, $\nu$ satisfy (A1) and $\F(\varphi_0) \in \mathrm{L}^1(\Omega)$. Then, $(\overline{\u}, \varphi)$ satisfies the following energy estimate:
    \begin{align}\label{mes}
    \mathcal{E}(\overline{\u}(t), \varphi(t)) &\leq C\Big(\mathcal{E}(\u_0, \varphi_0) +\int_{0}^t  \|\partial_t \h(\tau)\|_{\V^{\frac{1}{2}}(\Gamma)}^2d\tau +   \int_{0}^{t}\|\h(\tau)\|_{\V^{\frac{3}{2}}(\Gamma)}^2 d\tau\Big)\no\\&\quad\times\exp\Big(\int_0^t\|\h(\tau)\|^2_{\V^\frac{1}{2}(\Gamma)} d\tau\Big).
\end{align}
\end{lemma}
\begin{proof}
Multiplying $\eqref{7}_1$  by $\mu$ and $\eqref{7}_3$  by $\overline{\u}$ respectively and integrating over $\Omega$  and further adding them together yields 
\begin{align}\label{11}
&  \frac{d}{dt}\Big(\frac{1}{2} \|\overline{\u}\|^2 + \frac{1}{2}\|\nabla\varphi \|^2 +\int_\Omega F(\varphi)\Big) + (\nu(\varphi)-\nu_1) \| \nabla \overline{\u} \|^2 + \|\nabla\mu\|^2 \no \\ & =- b(\overline{\u}, \u_e, \overline{\u}) - b(\u_e, \u_e, \overline{\u}) - \langle \partial_t\u_e, \overline{\u} \rangle + \int_\Omega(\u_e \cdot \nabla \mu)\varphi. 
\end{align}
We now estimate each term on the right-hand side of equation \eqref{11} using Young's inequality, Poincar\'e inequality, Agmon's inequality, and the Sobolev embedding. 
We have the following estimates:
\begin{itemize}
\item $|b(\overline{\u}, \u_e, \overline{\u})|   \leq \|\overline{\u}\|_{\mathbb{L}^4} \, \|\nabla\u_e\| \,  \|\overline{\u}\|_{\mathbb{L}^4}  
   \leq C \|\nabla\u_e\|^2 \|\overline{\u}\|^2 + \frac{\nu}{4} \|\nabla \overline{\u}\|^2$,\\
   \item $|b(\u_e, \u_e, \overline{\u})|   \leq \| \u_e \|_{\mathbb{L}^\infty} \, \| \nabla\u_e \| \, \|\overline{\u} \| 
    \leq \frac{2}{C}\|\nabla\u_e\|^2 \|\overline{\u}\|^2 + \frac{C}{2}\|\u_e\|_{\V^2(\Omega)}^2 $,\\
    \item $\langle \partial_t\u_e, \overline{\u} \rangle  
     \leq C \|\partial_t\u_e\|^2 + \frac{\nu}{4} \|\nabla \overline{\u}\|^2  $,\\
     \item  $\Big|\int_\Omega(\u_e \cdot \nabla \mu)\varphi\Big| 
     \leq \|\u_e\|_{\mathbb{L}^4} \|\nabla\mu\| \, \|\varphi\|_{\mathrm{L}^4} 
     \leq \frac{1}{2} \|\nabla\mu\|^2 +\frac{1}{2} \|\u_e\|_{\mathbb{V}^1(\Omega)}^2\|\varphi\|_{\mathrm{H}^1}^2$.
  \end{itemize}
Considering the above estimates and using Remark \ref{nu}, we get from \eqref{11} the following : 
\begin{align}\label{en3}
  &\frac{d}{dt}\big( \|\overline{\u}\|^2  +  \|\nabla\varphi \|^2 +2\int_\Omega F(\varphi)\big) + \nu \|\nabla \overline{\u}\|^2 +   \|\nabla\mu\|^2 \no \\ & \quad\leq  C(\|\partial_t \u_e\|^2  + \|\u_e\|_{\V^2(\Omega)}^2) + C  \|\u_e\|_{\V^1(\Omega)}^2\big (\|\overline{\u}\|^2 + \|\nabla\varphi \|^2 +2\int_\Omega F(\varphi)\big).  
\end{align}
Now, integrating \eqref{en3} from $0$ to $t$ and using \eqref{estimate h} and \eqref{estimate ht} we get
\begin{align}\label{en4}
   & \mathcal{E}(\u(t), \varphi(t)) \leq \mathcal{E}(\u(0), \varphi(0)) +C\Big(\int_{0}^t  \|\partial_t \h(\tau)\|_{\V^{\frac{1}{2}}(\Gamma)}^2d\tau +   \int_{0}^{t}\|\h(\tau)\|_{\V^{\frac{3}{2}}(\Gamma)}^2 d\tau\Big)   \no\\& \quad +C\int_{0}^{t} \|\u_e(\tau)\|_{\V^1(\Omega)}^2 \mathcal{ E}(\u(\tau), \varphi(\tau)) d\tau.  
\end{align}
Then applying Gronwall's inequality in \eqref{en4} we obtain
\begin{align*}
   \mathcal{E}(t) :=  \mathcal{E}(\u(t), \varphi(t)) \leq &C\Big(\mathcal{E}(\u(0), \varphi(0)) +\int_{0}^t  \|\partial_t \h(\tau)\|_{\V^{\frac{1}{2}}(\Gamma)}^2d\tau +   \int_{0}^{t}\|\h(\tau)\|_{\V^{\frac{3}{2}}(\Gamma)}^2 d\tau\Big)\no\\&\quad \times\exp\Big(\int_0^t\|\h(\tau)\|^2_{\V^\frac{1}{2}(\Gamma)} d\tau\Big).
\end{align*}
\end{proof}
Now we state the existence result for the system \eqref{7}.
\begin{theorem}\label{thm4.2}
Let $\overline{\u}_0 \in \G_{\text{div}}$ and $\varphi_0 \in \mathrm{H}^1$. Let $\nu$ and $F$ satisfies (A1), (A2), (A4), (A7), and $\h$ satisfy Assumption \ref{H}. Moreover, let us assume $\u_0$ and $\h$ satisfies the compatibility condition
\begin{align}\label{compatibility}
   \u_0\big|_{\Gamma} = \h\big|_{t = 0}.
\end{align}
        Then, for any $T>0$, there exists a weak solution $(\overline{\u}, \varphi)$ to the system \eqref{7} in the sense of Definition \ref{weak sol def}.  
\end{theorem}
The existence of a global weak solution for the homogeneous system obtained in the Theorem \ref{thm4.2} guarantees a global weak solution $(\u,\varphi)$ for the non-homogeneous system \eqref{equ P}.
Hence we can prove the following  existence of a weak solution result for the nonhomogeneous system.
\begin{theorem}
Let $\u_0 \in \V^0(\Omega),$ $\varphi_0 \in \mathrm{H}^1$ and $T>0$. Let $\nu$ and $F$ satisfies (A1), (A2), (A4), (A7) from Assumption \ref{prop of F} and $\h$ satisfy Assumption \ref{H} and the compatibility condition \eqref{compatibility}.
Then, there exists a global  weak solution $(\u,\varphi)$ to the system \eqref{equ P} such that 
\begin{equation}\label{weak 3}
\left\{
\begin{aligned}
&	\u \in \mathrm{L}^{\infty}(0,T;\V^0(\Omega)) \cap \mathrm{L}^2(0,T;\V^1(\Omega)), \\
& \varphi \in \mathrm{L}^{\infty}(0,T;\mathrm{H}^1) \cap \mathrm{L}^2(0,T;\mathrm{H}^2), \\ 
	&\partial_t\u \in \mathrm{L}^{2}(0,T;{\V^1(\Omega)}'),  \\
		&\partial_t\varphi \in \mathrm{L}^{2}(0,T;({\mathrm{H}^1})'),  \\
		&\mu \in \mathrm{L}^2(0,T;\mathrm{H}^1).
\end{aligned}
\right.
\end{equation}
\end{theorem}
\begin{proof}
From the Theorem \ref{thm4.2}, we have a global weak solution $(\overline{\u},\varphi)$ for the homogeneous system \eqref{7}.  The Theorem \ref{elliptic lifting} proves the existence of $\u_e$. Hence, in the sense of Definition \ref{weak sol def}, we can prove the existence of a weak solution $ \u, \varphi$ for the system \eqref{equ P}.
\end{proof}
\begin{proof}\textbf{[Proof of Theorem \ref{thm4.2}:]}
We will prove the existence by a Semi-Galerkin approximation. We use the  Faedo-Galerkin approximation only for the concentration equation.
Hence we consider $\{\psi_k\}_{k \geq 1}$, which are  eigenfunctions of Neumann operator $B_1 = -\Delta + I$ in $\mathrm{H}^1$. Let $\Psi_n = \langle \psi_1,...,\psi_n \rangle$ be $n$-dimensional subspace, $ \overline{P}_n = P_{\Psi_n}$ be orthogonal projector of $\Psi_n$ in $\mathrm{L}^2$. We complete the proof in several steps. 

First we will find an approximate solution $\overline{\u}^n \in  \mathrm{L}^\infty(0, T; \G_{\text{div}}) \cap \mathrm{L}^2(0, T; \V_{\text{div}})$ and $\varphi^n \in \mathrm{H}^1(0, T; \Psi_n)$ that satisfies the following weak formulation:
\begin{align}\label{dist for u}
    \langle \partial_t\overline{\u}^n(t),&\mathbf{v} \rangle + (\nu(\varphi^n)-\nu_1) (\nabla \overline{\u}^n(t),\nabla\mathbf{v}) +  (((\overline{\u}^n(t) + \u_e(t)) \cdot \nabla) (\overline{\u}^n(t) + \u_e(t)),\mathbf{v})\no  \\  &= (\mu^n(t) \nabla \varphi^n(t),\mathbf{v}) 
     -\langle \partial_t\u_e(t), \mathbf{v} \rangle, \,   \text{ holds } 
     \forall \, \mathbf{v} \in \V_{\text{div}}, \text{ a.e.  } t \in  (0, T), 
\end{align}
and
\begin{align}\label{dist for phi}
    \langle \partial_t\varphi^n(t), \psi \rangle + \big((\overline{\u}^n(t) + \u_e(t)) \cdot\nabla \varphi^n(t), \psi \big) + (\nabla \mu^n(t
    ), \nabla\psi) = 0, \, \text{ holds }  \forall \, \psi \in \Psi_n, \text{ a.e. } t \in (0, T),
\end{align}
where 
\begin{align}\label{dist for mu}
    \mu^n = \overline{P}_n(-\Delta\varphi_n + F'(\varphi_n),
\end{align}
together with the initial conditions
\begin{equation}\label{initial cond}
\left\{
     \begin{aligned}
    \overline{\u}^n(x, 0) &=  \overline{\u}_0, \, \, \text{ in } \Omega,  \\
    \varphi^n(x, 0) &= \overline{P}_n\varphi_0 =: \varphi^n_0 \, \, \text{ in } \Omega.
\end{aligned}
\right.
\end{equation}
We will use Schauder's fixed-point theorem to prove the existence of approximate solutions. Since our equation is a coupled nonlinear equation, to set up a fixed point argument, we use the following splitting:
first, for a given vector,
\begin{center}
$\overline{\varphi}^n  \in C([0, T]; \Psi_n), $
we find
  $ \overline{\u}^n \in  \mathrm{L}^\infty(0, T; \G_{\text{div}}) \cap \mathrm{L}^2(0, T; \V_{\text{div}}), $
\end{center}
which satisfies \eqref{dist for u}. Then, for this $\overline{\u}^n $ we find $\varphi^n \in \mathrm{H}^1(0, T; \Psi_n)$ which satisfies \eqref{dist for phi}. This gives us approximate solution $(\overline{\u}^n, {\varphi}^n) $. Finally, using appropriate convergences, we can find the solution for \eqref{7}. 
More precisely, for given
\begin{align*}
\overline{\varphi}^n(\x, t) = \sum_{i = 1}^n c_i(t)\psi_i \,   \in C([0, T]; \Psi_n),  \text{ and } 
    \overline{\mu}^n(\x, t) = \sum _{i = 1}^n d_i(t) \psi_i \,   \in C([0, T]; \Psi_n),
    \end{align*}
we find 
\begin{align}
   \overline{\u}^n \in  \mathrm{L}^\infty(0, T; \G_{\text{div}}) \cap \mathrm{L}^2(0, T; \V_{\text{div}}), \no
\end{align}
which solves
    \begin{equation}\label{APFu}
\left\{    
    \begin{aligned} 
    &\langle \partial_t\overline{\u}^n(t),\mathbf{v} \rangle + (\nu(\overline{\varphi}^n)-\nu_1) (\nabla \overline{\u}^n(t),\nabla\mathbf{v}) +\big(((\overline{\u}^n(t) + \u_e(t)) \cdot \nabla) (\overline{\u}^n(t) + \u_e(t)),\mathbf{v}\big) \\ &\qquad\quad= (\overline{\mu}^n(t) \nabla \overline{\varphi}^n(t),\mathbf{v})  -\langle \partial_t\u_e(t), \mathbf{v} \rangle, \quad 
     \forall \mathbf{v} \in \V_{\text{div}}, \text{ a.e  } t \in  (0, T),  \,  \\
    &\overline{\u}^n(x, 0) =  \overline{\u}_0, \quad \text{ in } \Omega.
\end{aligned}
\right.
\end{equation}
Further for  solution $\overline{\u}^n \in  \mathrm{L}^\infty(0, T; \G_{\text{div}}) \cap \mathrm{L}^2(0, T; \V_{\text{div}})$ of \eqref{APFu},  we will find $\varphi^n \in \mathrm{H}^1(0, T; \Psi_n)$ that satisfies the following equations 
\begin{equation}\label{APFphi}
\left\{    
    \begin{aligned}
  \langle \partial_t\varphi^n, \psi \rangle + ( (\overline{\u}^n + \u_e) \cdot \nabla \varphi^n, \psi) &+ (\nabla \mu^n, \nabla\psi) = 0 ,  \, \, \text{ for all } \psi \in \Psi_n,   \\
   \mu^n = \overline{P}_n(-\Delta\varphi_n + F'(\varphi_n))\\
  \varphi^n(x, 0) = \varphi^n_0, \, \, \text{ in } \Omega.
\end{aligned}
\right.
\end{equation}
Here,  $c_i, \, d_i \, \in C([0, T])$ which implies,  
$$\sup_{t \in [0,T]} \sum_{i =1}^n |c_i(t)|^2 \leq M \text{ and } \sup_{t \in [0,T]} \sum_{i =1}^n |d_i(t)|^2 \leq M$$ for a positive number $M$, which depends on $\varphi_0 \text{ and } |\Omega|$. 
Therefore we have $$\sup_{t \in [0, T]}\|\overline{\varphi}^n(\cdot, t)\|^2 \leq M.$$
\textbf{Step 1:}
 For a given $\overline{\varphi}^n$ and $\overline{\mu}^n$ and for given $T> 0$ we find the solution of \eqref{APFu}, $\overline{\u}^n \in  \mathrm{L}^\infty(0, T; \G_{\text{div}}) \cap \mathrm{L}^2(0, T; \V_{\text{div}})$ by applying  \cite[Theorem 1.6]{barbu}. Now we derive an estimate of $\overline{\u}^n$ such that it continuously depends on given $\overline{\varphi}^n$, $\overline{\mu}^n$ and initial data $\overline{\u}_0$. So, taking test function $\mathbf{v}$ as $\overline{\u}^n$ in \eqref{APFu} we get
\begin{align}\label{w1}
    \frac{1}{2}\frac{d}{dt} \|\overline{\u}^n\|^2 + (\nu(\overline{\varphi}^n)-\nu_1)  \| \nabla \overline{\u}^n \|^2 &= -b(\overline{\u}^n, \u_e, \overline{\u}^n) - b(\u_e, \u_e, \overline{\u}^n) -  \int_{\Omega} \overline{\mu}^n (\nabla \overline{\varphi}^n \cdot\overline{\u}^n) \no\\&\quad- \langle \partial_t\u_e, \overline{\u}^n\rangle.
\end{align}
We estimate each term in R.H.S. of \eqref{w1}. By using integration by parts and H\"older inequality, we obtain
\begin{align}\label{w2}
 \Big|\int_{\Omega}\overline{\mu}^n (\nabla \overline{\varphi}^n\cdot \overline{\u}^n)\Big| & \leq \|\nabla\overline{\mu}^n\|\| \overline{\varphi}^n\|_{\mathrm{L}^4} \| \overline{\u}^n\|_{\mathbb{L}^4} \leq C M^2 |\Omega|^2 +\frac{\nu}{4} \|\nabla\overline{\u}^n \|^2. 
 \end{align}
 In the above, we have used the expression of $\overline{\varphi}^n, \, \overline{\mu}^n$ and the fact that $\psi_i$ are eigenfunctions of $B_1$.
Furthermore, we have some straightforward estimates on the trilinear operator 
\begin{itemize}
    \item $|b(\overline{\u}^n, \u_e, \overline{\u}^n)|   \leq \|\overline{\u}^n\|_{\mathbb{L}^4} \, \|\nabla\u_e\| \,  \|\overline{\u}^n\|_{\mathbb{L}^4} 
   \leq C\|\nabla \u_e\| ^2 \|\overline{\u}^n\|^2 + \frac{\nu}{2} \|\nabla \overline{\u}^n\|^2$,
   \item $|b(\u_e, \u_e, \overline{\u}^n)|   \leq \| \u_e \|_{\mathbb{L}^\infty} \, \| \nabla\u_e \|  \, \|\overline{\u}^n \| 
    \leq \frac{2}{C}\|\nabla\u_e\|^2 \|\overline{\u}^n\|^2 + \frac{C}{2}\|\u_e\|_{\V^2(\Omega)}^2 $.
\end{itemize}
Using Theorem \ref{elliptic lifting}, we also have 
\begin{itemize}
    \item$ \langle \partial_t\u_e, \overline{\u}^n\rangle   \leq \|\partial_t\u_e\| \, \|\overline{\u}^n\| 
     \leq \frac{1}{2} \|\partial_t\u_e\|^2 + \frac{1}{2} \|\overline{\u}^n\|^2 $
\end{itemize}
Now from \eqref{w1} and using Remark \ref{nu}, we obtain
\begin{align}\label{w6}
    \frac{1}{2}\frac{d}{dt} \|\overline{\u}^n\|^2 + \frac{\nu}{2} \| \nabla \overline{\u}^n \|^2  & \leq C M^2 |\Omega|^2 + \frac{
     1}{2}{\|\partial_t\u_e\|^2}+ C\|\u_e\|_{\V^2(\Omega)}^2+C(1 +  \|\nabla\u_e\|^2 )  \|\overline{\u}^n\|^2.
\end{align}
Integrating \eqref{w6} from $0$ to $t$ and using Gronwall's inequality finally  we deduce
\begin{align}\label{13}
   \|\overline{\u}^n\|^2 + \nu \int_{0}^{t}\| \nabla \overline{\u}^n \|^2  \leq C \Big(&\|\overline{\u}_0^n \|^2 +  \int_{0}^{t}M^2 |\Omega|^2 + \int_{0}^{t}\|\partial_t\h\|_{\V^{\frac{1}{2}}(\Gamma)}^2+\int_{0}^{t}\|\h\|_{\V^{\frac{3}{2}}(\Gamma)}^2 \Big) \no\\&\times \exp \big(Ct+C\int_{0}^{t} \|\h\|_{\V^{\frac{1}{2}}(\Gamma)}^2\big)  , \, \forall t \in [0,T),  
\end{align}
which implies that there exists a constant $L$ depending on $T, M $ and initial and boundary data such that
\begin{align}
\|\overline{\u}^n\|_{\mathrm{L}^\infty(0, T; \G_{\text{div}})\cap\mathrm{L}^2(0, T; \V_{\text{div}})} & < L(T, M).
\end{align}
 Let us define the operator $\I_n : C(0, T; \Psi_n) \longrightarrow \mathrm{L}^\infty(0, T; \G_{\text{div}}) \cap \mathrm{L}^2(0, T; \V_{\text{div}})$ given by $\I_n(\overline{\varphi}^n) = \overline{\u}^n$. Observe that $\I_n$ is continuous which   can be easily seen from \eqref{13} as  $\overline{\u}^n$ continuously depends on initial data and the fixed vector $\overline{\varphi}^n \text{ and } \overline{\mu}^n$. 

\textbf{Step 2: } Now we want to find $\varphi^n ( x, t) = \sum_{i = 1}^n \overline{c}_i(t) \psi_i$  and $\mu^n (x, t) = \sum_{i = 1}^n \overline{d}_i(t) \psi_i$ which satisfies  \eqref{APFphi} for the $\overline{\u}^n$, which we have determined in Step 1.
By taking $\psi = \psi_i$ in \eqref{APFphi}, we get a system of non-linear ODEs in $\overline{c}_i(t)$ given by
\begin{align}\label{w7}
    \overline{c}_i'(t) + \sum_{j = 1}^n \overline{c}_j(t) \int_\Omega (\overline{\u}^n \cdot \nabla)\psi_j \, \psi_i \, dx = 
    \overline{c}_i(t)^2  &+ \int_\Omega \psi_i \F'\big( \sum_{i = 1}^n \overline{c}_i(t) \psi_i\big)    - \sum_{j = 1}^n \overline{c}_j(t) \int_\Omega (\u_e \cdot \nabla)\psi_j \, \psi_i
\end{align}
 for all $ i=1, \cdots, n.$ Using the local Lipschitz property of the nonlinear terms and invoking Carath\'eodory's  existence theorem we ensure that \eqref{w7} admits a unique local solution in $[0, T_n)$ for some $T_n \in (0, T]$ such that $\overline{c}_i\in \mathrm{H}^1(0, T_n)$ ( $\overline{d}_i$ can be found by solving $\eqref{APFphi}_2$ ). Thus, we have $\varphi^n (t, x) \in \mathrm{H}^1(0, T_n; \Psi_n)$.

Now we will show that $\varphi^n$ is bounded in $\mathrm{L}^{\infty}(0,T_n;\mathrm{H}^1) \cap \mathrm{L}^{2}(0,T_n;\mathrm{H}^2)$. Taking the test function $\psi$ to be $\mu^n +\varphi^n$  in $\eqref{APFphi}_1$, we derive
\begin{align}\label{15b}
    \frac{d}{dt}\Big( \frac{1}{2}\|\varphi^n \|^2 + \frac{1}{2}\|\nabla\varphi^n \|^2 + &\int_\Omega F(\varphi^n)\Big) + \|\nabla\mu^n\|^2 = -( (\overline{\u}^n + \u_e) \cdot \nabla \varphi^n, \mu^n) \no \\
   & \quad -( (\overline{\u}^n + \u_e) \cdot \nabla \varphi^n, \varphi^n)-(\nabla \mu^n, \nabla\varphi^n).  
\end{align}  
Estimating the RHS of \eqref{15b}, we have
\begin{align*}
&\frac{d}{dt}\Big( \frac{1}{2}\|\varphi^n \|^2 + \frac{1}{2}\|\nabla\varphi^n \|^2 + \int_\Omega F(\varphi^n)\Big) + \|\nabla\mu^n\|^2 \no \\ & \leq C\|\overline{\u}^n\|_{\mathbb{L}^4} \|\nabla\mu^n\| \|\varphi^n\|_{\mathrm{L}^4} + C\|\u_e\|_{\V^2(\Omega)}\|\varphi^n\|(\|\nabla\mu^n\|+\|\nabla\varphi^n\|) + \frac{1}{4} \|\nabla\mu^n\|^2+ C\|\nabla\varphi^n\|^2 .
\end{align*}
After some straightforward calculations, we deduce
\begin{align}\label{15}
&\frac{d}{dt}\Big( \frac{1}{2}\|\varphi^n \|^2_{\mathrm{H}^1}  + \int_\Omega F(\varphi^n)\Big) + \frac{1}{4}\|\nabla\mu^n\|^2  \leq C (\|\nabla\overline{\u}^n\|^2 + \|\u_e\|^2_{\V^2(\Omega)}+ 1)\|\varphi^n\|_{\mathrm{H^1}}^2 
\end{align}
Integrating \eqref{15} from $0$ to $T_n$ we obtain 
\begin{align*}
  \|\varphi^n \|_{\mathrm{H^1}}^2 + 2\int_\Omega F(\varphi^n) &+\frac{1}{2}\int_{0}^{T_n}\|\nabla\mu^n\|^2 \leq \|\varphi^n_0\|_{\mathrm{H}^1}^2  +  2\int_\Omega F(\varphi^n_0)\\ &\quad+ C \int_{0}^{T_n}(\|\nabla\overline{\u}^n\|^2 + \|\u_e\|^2_{\V^2(\Omega)}+ 1) \,\|\varphi^n\|_{\mathrm{H^1}}^2 .  
\end{align*}
Now, we will show $\varphi^n \in \mathrm{L}^2(0, T_n; \mathrm{H}^2)$. We have 
\begin{align*}
   \frac{1}{4} \|\nabla\mu^n \|^2 + \|\nabla\varphi^n \|^2 \geq (\mu^n, -\Delta \varphi^n) \geq \|\Delta\varphi^n\|^2 - C_3 \|\nabla\varphi^n\|^2,
\end{align*}
which implies
\begin{align}\label{17}
    \frac{1}{2} \|\nabla\mu^n \|^2 & \geq \|\Delta\varphi^n\|^2 - (C_3 +1) \|\nabla\varphi^n\|^2.
\end{align}
Combining \eqref{15} and \eqref{17} we obtain
\begin{align}\label{17a}
   \frac{d}{dt}\Big(\|\varphi^n \|_{\mathrm{H}^1}^2  +2\int_\Omega F(\varphi^n)\Big) + \|\Delta\varphi^n\|^2 + \frac{1}{4} \|\nabla\mu^n \|^2 & \leq C (\|\nabla\overline{\u}^n\|^2 + \|\u_e\|^2_{\V^2(\Omega)} +  1) \, \|\varphi^n\|_{\mathrm{H}^1}^2 .   
\end{align}
Integrating \eqref{17a} from $0$ to $t$, $0\leq t \leq T_n$, and then using Gronwall inequality and the  estimate \eqref{13} to estimate RHS, yields 
\begin{align}\label{16}
    \|\varphi^n(t)\|_{\mathrm{H}^1}^2 + 2\int_\Omega F(\varphi^n) + \int_0^{t} \|\Delta\varphi^n\|^2 + \frac{1}{4} \int_0^{t} \|\nabla\mu^n \|^2 & \leq C\Big(\|\varphi^n_0\|_{\mathrm{H}^1}^2 + 2\int_\Omega F(\varphi_0^n)\Big) 
    \\ &  \quad\times\exp(T_n + L + \|\h\|_{\mathrm{L}^2(0, T; \V^\frac{3}{2}(\Gamma))})\no
\end{align}
 Let $M > 0$ be such that $\|\varphi_0^n\|_{\mathrm{H^1}}^2 +2\int_\Omega F(\varphi_0^n) < \frac{M}{R}$, where $R>0$ is a constant. Then from \eqref{13} and \eqref{16}, if we consider $\|\overline{\varphi}^n(t)\|_{\mathrm{H^1}}^2 < M$ for all $t \in [0, T_n]$, we obtain $\|\varphi^n(t)\|_{\mathrm{H^1}}^2 < M$ for all $t \in [0, T'_n]$ where $0 < T'_n < T_n$ is sufficiently small.\\ 
From \eqref{16}, it easy to see that $\varphi^n$  depends continuously on initial data $\varphi_0$ and fixed $\overline{\u}^n$. Therefore by solving \eqref{APFphi} we get another continuous operator $\J_n: \mathrm{L}^\infty(0, T_n; \G_{\text{div}}) \cap \mathrm{L}^2(0, T_n; \V_{\text{div}}) \longrightarrow  \mathrm{H^1}(0, T_n; \Psi_n) $ such that $\J_n(\overline{\u}^n) = \varphi^n$.

\textbf{Step 3: }Therefore, the composition map 
\begin{align*}
    \J_n \circ \I_n: C([0, T_n]; \Psi_n) \longrightarrow \mathrm{H^1}(0, T_n; \Psi_n), \, \, \J_n \circ \I_n(\overline{\varphi}^n) = \varphi^n,
\end{align*}
is continuous. Since $\Psi_n$ is a finite dimensional space, compactness of $\mathrm{H^1}(0, T_n; \Psi_n)$ into $C([0, T_n]; \Psi_n)$ gives that $\J_n \circ \I_n$ is a compact operator from $C([0, T_n]; \Psi_n)$ to itself. We can choose $T'_n \in (0, T_n)$ such that $\sup_{t \in [0, T'_n]}\|\varphi^n(t)\|_{\mathrm{H^1}}^2 < M.$ Then by Schauder's fixed point theorem there exists a fixed point $\varphi^n$ in the set
\begin{align*}
   \Big \{\varphi^n \in C([0, T'_n]; \Psi_n) \, | \, \sup_{t \in [0, T'_n]}\|\varphi^n(t)\|_{\mathrm{H^1}}^2 < M \, \text{ with } \varphi^n(0) =\overline{P}_n\varphi_0 \Big\},
\end{align*}
such that $\varphi^n \, \in \, \mathrm{H^1}(0, T'_n; \Psi_n) \text{ and } \overline{\u}^n \, \in \, \mathrm{L}^\infty(0, T'_n; \G_{\text{div}}) \cap \mathrm{L}^2(0, T'_n; \V_{\text{div}})$ and satisfies the variational formulation
\begin{equation}\label{as}
\left\{    
    \begin{aligned}
&\langle \partial_t\overline{\u}^n(t),\mathbf{v} \rangle + (\nu(\varphi^n)-\nu_1) (\nabla \overline{\u}^n(t),\nabla\mathbf{v}) +\big(((\overline{\u}^n(t) + \u_e(t)) \cdot \nabla) (\overline{\u}^n(t) + \u_e(t)),\mathbf{v}\big) \\ & \quad= (\mu^n(t) \nabla \varphi^n(t),\mathbf{v})  -\langle \partial_t\u_e(t), \mathbf{v} \rangle, \quad 
     \forall \mathbf{v} \in \V_{\text{div}}, \text{ a.e  } t \in  (0, T'_n),\\
    & \langle \partial_t\varphi^n(t), \psi \rangle + ( (\overline{\u}^n + \u_e)(t) \cdot \nabla \varphi^n(t), \psi) + (\nabla \mu^n(t), \nabla\psi) = 0 ,  \, \, \text{ for all } \psi \in \Psi_n, \text{ a.e  } t \in  (0, T'_n).   
\end{aligned}
\right.
\end{equation}

\textbf{Step 4: } Now we want to show that $T_n'=T$.  We know that $(\overline{\u}^n, \varphi^n)$ satisfies \eqref{as}. It is easy to see that $(\overline{\u}^n, \varphi^n)$ satisfy the energy estimate \eqref{mes}. So, we have
\begin{align}\label{4.22}
    \|\varphi^n &\|_{\mathrm{H^1}}^2  + \|\overline{\u}^n\|^2  + \frac{\nu}{2} \int_0^{T_n'} \| \nabla \overline{\u}^n \|^2 \no  + \int_0^{T_n'}\|\Delta\varphi^n\|^2 + \frac{1}{2}\int_0^{T_n'} \|\nabla\mu^n\|^2  \no \\& \leq C(\|\h\|_{\mathrm{L}^2(0, T;\V^\frac{3}{2}(\Gamma))}, \|\partial_t \h\|_{\mathrm{L}^2(0, T;\V^{\frac{1}{2}}(\Gamma))}, \|\varphi_0\|_{\mathrm{H}^1}, \|\overline{\u}_0\|).
\end{align}
Note that the right-hand side of \eqref{4.22} is independent of $n$. Thus, we get the uniform bound for the approximate solution $(\overline{\u}^n, \varphi^n)$.  Thus, we can extend the solution to [0, T], and $\overline{\u}^n$, ${\varphi}^n$ satisfies
\begin{align}
    &\overline{\u}^n \text{ uniformly bounded in } \mathrm{L}^\infty(0, T; \G_{\text{div}})\cap \mathrm{L}^2(0, T; \V_{\text{div}}), \label{UEFu} \\
 &\varphi^n \text{ uniformly bounded in } \mathrm{L}^\infty(0, T;\mathrm{H}^1) \cap \mathrm{L}^2(0, T; \mathrm{H}^2), \label{UEFphi} \\
& \mu^n \text{ uniformly bounded in }  \mathrm{L}^2(0, T; \mathrm{H}^1). \label{UEFmu} 
\end{align}  

\textbf{Step 5: }Now we will estimate $\partial_t\overline{\u}^n \in \mathrm{L}^2(0, T; \V_{\text{div}}')$ and $\partial_t\varphi^n\in \mathrm{L}^2(0, T; ({\mathrm{H}^1})')$. For that, let us first recall that $\{\mathbf{e}_j\}$ are eigenfunctions of $\A$ and define $n$ dimensional subspace $\H_n = \langle \mathbf{e}_1, \mathbf{e}_2, \cdots, \mathbf{e}_n\rangle $ and consider the orthogonal projection of the space $\G_{\text{div}}$ on $\H_n$ defined by  $P_n:=P_{\H_n}$. Then
we can write \eqref{APFu} as
\begin{align}
    \partial_t\overline{\u}^n + P_n(\text{div}((\nu(\varphi)-\nu_1) \mathrm{D}\overline{\u}^n) + \B(\overline{\u}^n, \overline{\u}^n) + \B_1(\overline{\u}^n, \u_e)+\B_2(\u_e, \overline{\u}^n) + \B_3(\u_e, \u_e)\no\\ - \mu^n  \nabla \varphi^n + \partial_t\u_e) = 0. \label{ut projection} 
\end{align}

We multiply the equation \eqref{ut projection} by  $\overline{\u} = P_n\overline{\u} + (I - P_n)\overline{\u}$  and integrate by parts and use the fact that $P_n\overline{\u}$ is orthogonal to $(I - P_n)\overline{\u}$ and estimate every term in \eqref{ut projection} as follows:
\begin{itemize}
    \item $\|P_n(\text{div}((\nu(\varphi)-\nu_1) \mathrm{D}\overline{\u}^n)\|_{\V_{\text{div}}^{'}}  \leq C(\nu_2-\nu_1) \|\overline{\u}^n\|_{\V_{\text{div}}}\label{stokes projection}$, \\
    \item $\|P_n\B_1(\overline{\u}^n, \u_e)\|_{\V_{\text{div}}^{'}}    \leq C \|\overline{\u}^n\|_{\V_{\text{div}}} \|\nabla\u_e\|$,\\
    \item $\|P_n\B_2(\u_e, \overline{\u}^n)\|_{\V_{\text{div}}^{'}}\leq (\nu_2-\nu_1)\|\nabla\ou^n\|^2_{\V_{\text{div}}}+C\|\u_e\|^4_{\V^1(\Omega)}\|\ou^n\|^2,$\\
    \item $\|P_n\B_3(\u_e, \u_e)\|_{\V_{\text{div}}^{'}}  
    \leq C \|\u_e\|^2_{\V^1(\Omega)}$,\\
    \item $\|P_n(\mu^n  \nabla \varphi^n)\|_{\V_{\text{div}}'}   \leq \|\varphi^n\|_{\mathrm{H}^1} \, \|\nabla\mu^n\|$,\\
    \item $\|P_n(\partial_t\u_e)\|_{\V_{\text{div}}'}  \leq C \|\partial_t\u_e\|$ . 
\end{itemize}
Substituting these estimates in \eqref{ut projection} we get the bound on the time derivative of $\overline{\u}^n$ as 
\begin{align}
   &\int_{0}^{T} \|\partial_t\overline{\u}^n(t)\|^2_{\V'_{\text{div}}} \leq C \Big[(\nu_2-\nu_1)\int_{0}^{T} \|\overline{\u}^n(t)\|^2_{\V_{\text{div}}} dt  +\int_{0}^{T}\|\h(t)\|_{\V^{\frac{1}{2}}(\Gamma)} ^2 dt  \label{UEFut}
   \\ &  +\sup_{t \in [0, T]}\|\varphi^n(t)\|^2_{\mathrm{H}^1} \int_{0}^{T} \|\nabla\mu^n(t)\|^2 dt+\sup_{t \in [0, T]} \|\u_e(t)\|^2_{\V^1(\Omega)} \Big(\int_{0}^{T}\|\overline{\u}^n(t)\|^2_{\V_{\text{div}}} dt + \int_{0}^{T}\|\u_e(t)\|^2_{\V^2(\Omega)} dt\Big)\Big].\no   
\end{align}
Similarly, we can write \eqref{APFphi} as 
\begin{align}
 \partial_t \varphi^n + \overline{P}_n \Big( (\overline{\u}^n + \u_e) \cdot \nabla \varphi^n - \Delta\mu^n\Big) = 0.\label{phit projection} 
\end{align}
Let us consider $\psi = \psi_1 + \psi_2$  where $\psi_1 = \overline{P}_n\psi = \sum_{i = 1}^{n} (\psi, \psi_i)\psi_i \in \Psi_n$ and $\psi_2 = (I - \overline{P}_n)\psi  = \sum_{i = n}^{\infty} (\psi, \psi_i)\psi_i \in \Psi_n^{\perp}$ as a test function. Now taking duality pairing of \eqref{phit projection} with $\psi$ and using the fact that $\psi_1$ and $\psi_2$ are orthogonal in  the space $\mathrm{V}$, we obtain
\begin{itemize}
    \item $\Big|\int_{\Omega}\overline{P}_n(\overline{\u}^n \cdot \nabla\varphi^n)\psi_1\Big|\leq C \|\overline{\u}^n\|_{\V_{\text{div}}} \|\nabla \varphi^n\|\|\psi\|_{\mathrm{H}^1}$, \\
    \item $\Big|\int_{\Omega}\overline{P}_n(\u_e \cdot \nabla \varphi^n)\psi_1\Big|  \leq C \|\u_e\|_{\V^1(\Omega)} \|\nabla\varphi^n\|\|\psi\|_{\mathrm{H}^1}$, \\
    \item $\Big|\int_{\Omega} \overline{P}_n(\Delta \mu^n) \psi_1\Big|  \leq C \|\nabla \mu^n\| \, \|\psi\|_{\mathrm{H}^1}$.
\end{itemize}
Therefore, from \eqref{phit projection} we have
\begin{align}
   \int_{0}^{T} \|\partial_t\varphi^n(t)\|^2_{({\mathrm{H}^1})'} & \leq C \, \sup_{t \in [0, T]} \|\varphi^n(t)\|^2_{\mathrm{H}^1} \int_{0}^{T}\big(\|\overline{\u}^n(t)\|^2_{\V_{\text{div}}} +  \|\u_e(t)\|^2_{\V^2(\Omega)}\big) dt + \int_{0}^{T}\|\nabla \mu^n(t)\| dt . \label{UEFphit}
\end{align}
\textbf{Step 6: }From \eqref{UEFu}-\eqref{UEFmu}, \eqref{UEFut}, \eqref{UEFphit} and by using Banach-Alaoglu theorem we can extract subsequences $\overline{\u}^n$, $\varphi^n$  and $\mu^n$ that satisfy
\begin{equation}\label{3.39}
\left\{
\begin{aligned}
    \overline{\u}^n & \rightharpoonup \overline{\u} \text{ in } L^\infty(0, T; \G_{\text{div}}) \text{ weak - }\ast,  \\
   \overline{\u}^n & \rightharpoonup \overline{\u} \text{ in } L^2(0, T; \V_{\text{div}}) \text{ weak },\\
   \partial_t\overline{\u}^n & \rightharpoonup \partial_t\overline{\u} \text{ in } L^2(0, T; \V'_{\text{div}}) \text{ weak }, \\
    \overline{\u}^n & \rightarrow \overline{\u} \text{ in } L^2(0, T; \G_{\text{div}}) \text{ strong }, \\
   \end{aligned}
\right.
\end{equation}
\begin{equation}\label{3.40}
\left\{
\begin{aligned}
   \varphi^n & \rightharpoonup \varphi \text{ in } L^\infty(0, T; \mathrm{H^1}) \text{ weak - }\ast, \\
   \varphi^n & \rightharpoonup \varphi \text{ in } L^2(0, T; \mathrm{H}^2) \text{ weak }, \\
 \partial_t \varphi^n & \rightharpoonup \partial_t\varphi \text{ in } L^2(0, T; ({\mathrm{H}^1})') \text{ weak }, \\
   \mu^n & \rightharpoonup \mu \text{ in } L^2(0, T; \mathrm{H^1}) \text{ weak },\\
    \varphi^n & \rightarrow \varphi \text{ in } L^2(0, T; \mathrm{H^1}) \text{ strong }.
\end{aligned}
\right.
\end{equation}
\textbf{Step 7: }With these convergences established, our primary objective is to take the limit in the equations \eqref{as}. The process of passing to limit in the linear terms of \eqref{as} is straightforward. Additionally, leveraging the strong convergence outlined in \eqref{3.39}-\eqref{3.40}, the process of passing to the limit in the nonlinear terms follows standard procedures, as detailed comprehensively in \cite{weak}. Since $\overline{\u} \in \mathrm{L}^2(0, T; \V_{\text{div}})$ and $\overline{\u}_t  \in \mathrm{L}^2(0, T; \V'_{\text{div}})$, it implies $\overline{\u} \in C([0, T]; \G_{\text{div}})$ using Aubin-Lions' compactness lemma. Similarly, $\varphi \in  \mathrm{L}^\infty(0, T; \mathrm{H^1})$ and $\partial_t\varphi\in  \mathrm{L}^2(0, T; ({\mathrm{H}^1})')$, which gives $\varphi \in C([0, T]; \mathrm{H})$. The initial condition also satisfies in the weak sense since $\overline{\u}$ and $\varphi$ are right continuous at 0.
\end{proof}

 Several remarks are in order, derived from the properties of the weak solution.
\begin{remark}\label{L4esti_phi}
    Note that if $(\u, \varphi)$ is a weak solution of the system \eqref{equ P} and additionally $F$ satisfy (A5) then $\varphi\in \mathrm{L}^4(0, T;\mathrm{H}^2)\cap\mathrm{L}^2(0,T;\mathrm{H}^3)$ for any $T>0$, indeed, we have the following estimate:
     \begin{align*} \int_0^T\|\Delta\varphi(t)\|^4 dt &\leq \int_0^T\|\nabla\mu(t)\|^2 \|\nabla\varphi(t)\|^2 dt + C_3\int_0^T\|\nabla\varphi(t)\|^2 dt\\
      & \leq \|\nabla\mu\|^2_{\mathrm{L}^2(0, T; \mathrm{L}^2)}\|\varphi\|^2_{\mathrm{L}^\infty(0, T; \mathrm{H}^1)} + C_3 \|\varphi\|^4_{\mathrm{L}^\infty(0, T; \mathrm{H}^1)}.
  \end{align*}
   By the elliptic regularity theory one can conclude that $\varphi \in \mathrm{L}^4(0, T; \mathrm{H}^2).$ \\
  To see $\varphi\in \L^2(0,T; \mathrm{H}^3)$, taking the formal gradient of  $\eqref{7}_2$ we get, 
\begin{align*}
    \nabla \mu= -\nabla \Delta \varphi + \nabla F'(\varphi) .
\end{align*}Taking the $\mathrm{L}^2$-norm on both sides of the above equation, we further get,
\begin{align}\label{e6}
     \|\nabla (\Delta\varphi)\|^2 & \leq \|\nabla \mu\|^2 + \|\nabla (F'(\varphi)\|^2 .
     \end{align}
     Now, from (A5) and Young's inequality, we have,
    \begin{align}\label{e7}
    \|\nabla F'(\varphi)\|^2  & = \|F''(\varphi) \nabla \varphi \|^2 =  \int_\Omega (C_4|\varphi|^{2q-2}+C_4') |\nabla \varphi|^2 \no \\
    & \leq C_4 \int_\Omega |\varphi|^{2q-2} |\nabla\varphi|^2 + C_4' \|\nabla \varphi\|^2 \no \\
    & \leq C_4 \|\varphi\|^{2q-2}_{\L^{2q-4}}\|\nabla\varphi \|^{2}_{\L^{2q-2}} + C_4' \|\nabla \varphi\|^{2}  \no \\
    & \leq C \|\varphi\|^{4q-4}_{\mathrm{H}^1} + C \|\varphi \|^4_{\mathrm{H}^2} + C_4' \|\nabla \varphi\|^2\no \\
    & \leq C (\| \varphi\|^4_{\mathrm{H}^2} + \|\nabla \varphi\|^2+\|\varphi\|^{4q-4}_{\mathrm{H}^1} ).
     \end{align}
  In the above, we have used the Sobolev embedding $\mathrm{H}^1\hookrightarrow\mathrm{L}^p$,  for all $1< p < \infty$. Now, by substituting \eqref{e7} in \eqref{e6} we get
     \begin{align*}
      \|\nabla (\Delta\varphi)\|^2 & \leq \|\nabla \mu\|^2  + C (\| \varphi\|^{2}_{\mathrm{H}^2} + \|\nabla \varphi\|^2 ).
     \end{align*}
     Integrating both sides from $0$ to $T$ we obtain,
     \begin{align*}
         \int_0^T \|\nabla (\Delta\varphi)\|^2 \leq \int_0^T \|\nabla \mu\|^2  + C\int_0^T \big(\| \varphi\|^{4}_{\mathrm{H}^2}+\|\nabla \varphi\|^2+\|\varphi\|^{4q-4}_{\mathrm{H}^1}\big) \leq C.
     \end{align*}
\end{remark}
\begin{remark}\label{mean_esti_phi}
   We note here that, unlike the zero Dirichlet boundary case, the mean $\varphi$ may not be constant $(\text{ for example }\langle\varphi(t)\rangle=\langle\varphi(0)\rangle, \, \text{ for all } t\geq 0).$ But we have $\langle\varphi(t)\rangle$ to be bounded for a.e. $t\geq 0$, indeed, from $\eqref{7}_1$ and integrating by parts we obtain
   \begin{align*}
      \frac{d}{dt}\langle\varphi(t)\rangle=\frac{1}{|\Omega|}\int_\Omega\frac{d\varphi(t)}{dt}= -\frac{1}{|\Omega|}\int_{\Gamma}(\h(t)\cdot n)\varphi(t) \, dS. 
   \end{align*}
   Then integrating from $0$ to $t$ and recalling that $\h$ satisfies Assumption \ref{H} and $(\u, \varphi)$ is corresponding weak solution we have
   \begin{align*}
       |\langle\varphi(t)\rangle|\leq &C_\Omega\big(\langle\varphi(0)\rangle+\|\h(t)\|_{\V^{-\frac{1}{2}}(\Gamma)}\|\varphi(t)\|_{\mathrm{H}^\frac{1}{2}(\Gamma)}\big)\\
       \leq &  C_\Omega\big(\langle\varphi(0)\rangle+\|\h(t)\|_{\V^{\frac{1}{2}}(\Gamma)}\|\varphi(t)\|_{\mathrm{H}^1(\Omega)}\big) < K,
   \end{align*}
   for a.e. $t\in (0, +\infty)$.
\end{remark}
\begin{remark}
  Similarly, if $F$ satisfy (A3) we have $\langle\mu(t)\rangle$ is bounded for a.e. $t\geq 0$, indeed the following estimate holds:
  \begin{align*}
      |\langle\mu(t)\rangle|\leq \frac{1}{|\Omega|}\int_{\Omega}|F'(\varphi(t))| \, dx\leq C\|\varphi(t)\|^q_{\mathrm{H}^1}+C_2 < K_1, \text{ for a.e. $t\geq 0$}.
  \end{align*}
\end{remark}

\section{Continuous dependence} In this section, we discuss the continuous dependence of the solution on the initial and boundary data and deduce the uniqueness of the weak solution of the system \eqref{equ P}. We use the idea of \cite[Theorem 3.1]{GMT}, however, since we do not have a mean zero condition for the concentration equation, we cannot use $A_0^{-1}\varphi$ as a test function. Therefore, we have to modify the idea of \cite{GMT} to apply in our case.

Let $(\u_1, \varphi_1)$ and $(\u_2, \varphi_2)$ be two weak solutions of the system \eqref{equ P} with non-homogeneous boundaries $\h_1$ and $\h_2$ and initial conditions $\u_{i0}$, $\varphi_{i0}$, for $i=1,2$ respectively. Let  $\u_{e_1} \text{ and } \u_{e_2}$ be the corresponding lifting functions respectively. Then, $\overline{\u}_1=\u_1 - \u_{{e_1}}$, and $\overline{\u}_2=\u_2 - \u_{{e_2}}$ are the solutions of the system \eqref{7} corresponding to initial data $(\ou_{10}$, $\varphi_{10}), \, (\ou_{20}$, $\varphi_{20})$ respectively. Let us denote the differences of solutions by $\delta\overline{\u} = \overline{\u}_1 - \overline{\u}_2, \, \delta\varphi = \varphi_1 - \varphi_2$, $\delta\mu= \mu_1-\mu_2$, $\delta\pi=\overline{\pi}_{1}-\overline{\pi}_{2}$ and the nonhomogeneous boundary $\delta\h = \h_1 - \h_2$. Note that $\delta\u_e := \u_{e_1} - \u_{e_2}$ satisfies equation \eqref{equ1} with the boundary data $\delta\h$. Moreover, we set $\delta\varphi_0 = \varphi_{10} - \varphi_{20},  \,  \delta\ou_0 = \ou_{10} - \ou_{20}$. 
\begin{proposition}
   Let $\nu, \, F$ satisfies (A1)-(A5) and (A7). Also, let $(\u_1, \varphi_1)$ and $(\u_2, \varphi_2)$ be two weak solutions of the system \eqref{equ P} with non-homogeneous boundary data $\h_1$ and $\h_2$ respectively. Then, with the notation mentioned above,  for any $T>0$, the difference $(\delta\overline{\u}, \delta\varphi)$ satisfy the following inequality:
   \begin{align}\label{main_fde}
\|\delta\ou\|^2_{\sharp}+\|\delta\varphi\|^2+&\nu_1\int_0^T\|\delta\ou\|^2+\frac{1}{2}\int_0^T\|\Delta(\delta\varphi)\|^2  \leq M_1\Big[\|\delta\ou_0\|^2_{\sharp}+\|\delta\varphi_0\|^2+|\langle\delta\varphi\rangle|^2\no\\&\quad+\|\delta\h\|^2_{\mathrm{H}^1(0,T;\V^{\frac{1}{2}}(\Gamma)} +\|\delta\h\|^2_{\mathrm{L}^2(0, T;\V^\frac{3}{2}(\Gamma)) \cap \mathrm{L}^\infty(0, T; \V^\frac{1}{2}(\Gamma))}\Big], 
   \end{align}
   where  $M_1$ is a positive constant depending on norms of the weak solutions $(\ou_1,\varphi_1)$, $(\ou_2, \varphi_2)$ and corresponding boundary data $\h_1, \h_2$. 
    Recall the norm $\|\cdot\|_{\sharp}$ as defined in \eqref{norm}. \end{proposition}
\begin{proof}
 From the definition \ref{weak sol def}, $(\delta\overline{\u}, \delta\varphi)$ obeys the following weak formulation:
\begin{align}
 \langle\partial_t(\delta\varphi), \psi\rangle + &(\delta\ou \cdot \nabla \varphi_1, \psi) + (\ou_2 \cdot \nabla (\delta\varphi), \psi) + (\delta\u_e \cdot \nabla \varphi_1, \psi) + (\u_{e_2} \cdot \nabla(\delta \varphi), \psi)  \no\\&+ (\nabla(\delta\mu), \nabla\psi)=0, \, \, \forall\; \psi\in\mathrm{H}^1, \, \text{ a.e. } t\in (0, T)\label{20} 
  \end{align}
\begin{align}
&\langle\partial_t(\delta\ou),\v\rangle  +(\nu(\varphi_1)\mathrm{D}\ou_1, \nabla\v)- (\nu(\varphi_2) \mathrm{D}\ou_2, \nabla\v) - (\delta\ou\otimes\ou_1, \nabla\v) - (\delta\u_{e} \otimes\ou_1,\nabla\v) \no\\&\quad- (\delta\ou\otimes\u_{e_1}, \nabla\v) + ((\delta\u_e \cdot \nabla) \u_{e_1},\v)- (\ou_2 \otimes\delta\ou,\nabla, \v) - (\u_{e_2} \otimes\delta\ou,\nabla\v) - (\ou_2 \otimes\delta\u_e,\nabla\v) \no\\&\quad+((\u_{e_2}\cdot\nabla)\delta\u_e,\v)  = (\nabla(\delta\varphi)\otimes\nabla \varphi_2,\nabla\v) + (\nabla\varphi_1\otimes \nabla(\delta\varphi),\nabla\v)- \langle\partial_t(\delta\u_e), \v\rangle,  \label{20a}
\end{align}
$\forall\v\in\V_{\text{div}},$ and a.e. $t\in (0,T)$ where $$\delta\mu = -\Delta(\delta\varphi) + F'(\varphi_1) - F'(\varphi_2).$$  Note that, we have used the equalities $$(\u\cdot\nabla)\v=\text{div}(\u\otimes\v)$$ and $$\mu\nabla\varphi=\nabla\Big(\frac{1}{2}|\nabla\varphi|^2+F'(\varphi)\Big)-\text{div}(\nabla\varphi\otimes\nabla\varphi).$$
Now, taking   $\psi=\delta\varphi$ in \eqref{20}, we deduce
 \begin{align}
     \frac{1}{2}\frac{d}{dt}\|\delta\varphi\|^2 +(\nabla(\delta\mu),\nabla(\delta\varphi)) &= -(\delta\ou \cdot \nabla \varphi_1,\delta\varphi) - (\ou_2 \cdot \nabla (\delta\varphi),\delta\varphi)- (\delta\u_e \cdot \nabla \varphi_1, \delta\varphi) \no\\&\quad-((\u_{e_2} \cdot \nabla(\delta \varphi),\delta\varphi).\label{22}
     \end{align}
    Using the expression of $\delta\mu$ we obtain 
     \begin{align}\label{22a}
         \frac{1}{2}\frac{d}{dt}\|\delta\varphi\|^2 +\|\Delta(\delta\varphi)\|^2=&-((F'(\varphi_1)-F'(\varphi_2)), \Delta(\delta\varphi))-(\delta\ou \cdot \nabla \varphi_1,\delta\varphi) - (\ou_2 \cdot \nabla (\delta\varphi),\delta\varphi)\no\\&\quad- (\delta\u_e \cdot \nabla \varphi_1, \delta\varphi)-((\u_{e_2} \cdot \nabla(\delta \varphi), \delta\varphi).
     \end{align}
     Taking $\v=\A^{-1}(\delta\ou)$ in \eqref{20a}, we find
\begin{align}
    &\frac{1}{2}\frac{d}{dt} \|\delta\ou\|^2_{\sharp}+ (\nu(\varphi_1)\mathrm{D}(\delta\ou),\nabla\A^{-1}(\delta\ou))= ((\nu(\varphi_1)-\nu(\varphi_2) \mathrm{D}\overline{\u}_2,{\nabla\A^{-1}(\delta\ou)})+(\delta\ou\otimes\ou_1, \nabla\A^{-1}(\delta\ou))  \no\\&\quad + (\delta\u_{e} \otimes\ou_1,\nabla\A^{-1}(\delta\ou)) + (\delta\ou\otimes\u_{e_1}, \nabla\A^{-1}(\delta\ou)) - ((\delta\u_e \cdot \nabla) \u_{e_1},\A^{-1}(\delta\ou))+ (\ou_2 \otimes\delta\ou,\nabla \A^{-1}(\delta\ou))  \no\\&\quad+ (\u_{e_2} \otimes\delta\ou,\nabla\A^{-1}(\delta\ou)) + (\ou_2 \otimes\delta\u_e,\nabla\A^{-1}(\delta\ou))-((\u_{e_2}\cdot\nabla)\delta\u_e,\A^{-1}(\delta\ou))\no\\&\quad+ (\nabla(\delta\varphi)\otimes\nabla \varphi_2,\nabla\A^{-1}(\delta\ou)) + (\nabla\varphi_1\otimes \nabla(\delta\varphi),\nabla\A^{-1}(\delta\ou))- \langle\partial_t(\delta\u_e), \A^{-1}(\delta\ou)\rangle. \label{21} %
   \end{align}
   By the properties of Stokes operator [cf. \cite[Appendox B]{GMT}], there exists $\delta\pi\in\L^2(0, T, \mathrm{H}^1)$ such that $-\Delta\A^{-1}(\delta\ou)+\nabla(\delta\pi)=\delta\ou$ a.e. in $\Omega\times(0,T)$. Then, following a similar calculation from \cite[Theorem 3.1, (3.9)-(3.11)]{GMT} we write 
   \begin{align}\label{4.7}
      (\nu(\varphi_1)\mathrm{D}(\delta\ou),\nabla\A^{-1}(\delta\ou))\geq\nu_1\|\delta\ou\|^2-(\delta\ou,\nu'(\varphi_1)\mathrm{D}\A^{-1}(\delta\ou)\nabla\varphi_1)+ \frac{1}{2}(\nu'(\varphi_1)\nabla\varphi_1\cdot\delta\ou, \delta\pi).
   \end{align} 
Using \eqref{4.7} in  \eqref{21}, and adding it to \eqref{22a}, after rearranging the terms, we obtain
   \begin{align}\label{4.8}
      &\frac{1}{2}\frac{d}{dt} \big(\|\delta\ou\|^2_{\sharp}+\|\delta\varphi\|^2\big)+\nu_1\|\delta\ou\|^2+\frac{1}{2}\|\Delta(\delta\varphi)\|^2 \leq ((\nu(\varphi_1)-\nu(\varphi_2) \mathrm{D}\overline{\u}_2), \nabla\A^{-1}(\delta\ou)) \no\\&+\big(\delta\ou\otimes\ou_1+\ou_2 \otimes\delta\ou, \nabla\A^{-1}(\delta\ou)\big)+(\nabla(\delta\varphi)\otimes\nabla \varphi_2+\nabla\varphi_1\otimes \nabla(\delta\varphi),\nabla\A^{-1}(\delta\ou))\no\\  &-(\delta\ou,\nu'(\varphi_1)\mathrm{D}\A^{-1}(\delta\ou)\nabla\varphi_1)+ \frac{1}{2}(\nu'(\varphi_1)\nabla\varphi_1\cdot\delta\ou, \delta\pi)- (\ou_2 \cdot \nabla (\delta\varphi),\delta\varphi)-(\delta\ou \cdot \nabla \varphi_1, \delta\varphi)\no\\&- (\delta\u_e \cdot \nabla \varphi_1, \delta\varphi) -((\u_{e_2} \cdot \nabla(\delta \varphi),\delta\varphi)+(\delta\u_{e} \otimes\ou_1,\nabla\A^{-1}(\delta\ou)) + (\delta\ou\otimes\u_{e_1}, \nabla\A^{-1}(\delta\ou)) \no\\&- ((\delta\u_e \cdot \nabla) \u_{e_1},\A^{-1}(\delta\ou))+ (\u_{e_2} \otimes\delta\ou,\nabla\A^{-1}(\delta\ou)) + (\ou_2 \otimes\delta\u_e,\nabla\A^{-1}(\delta\ou))\no\\&-((\u_{e_2}\cdot\nabla)\delta\u_e,\A^{-1}(\delta\ou))-\langle\partial_t(\delta\u_e), \A^{-1}(\delta\ou)\rangle 
      +((F'(\varphi_1)-F'(\varphi_2)), \Delta(\delta\varphi))=\sum_{i=1}^{17}\mathcal{I}_i. 
   \end{align}
 We estimate the terms on RHS one by one. We observe that $\mathcal{I}_2$, $\mathcal{I}_4$, and $\mathcal{I}_5$ take exactly the same form as those in \cite[Theorem 3.1]{GMT}. Hence, directly adopting those estimates we get
  \begin{itemize}
      \item $|\mathcal{I}_2|\leq\frac{\nu_1}{12}\|\delta\ou\|^2+C(\|\ou_1\|^2_{\V_{\text{div}}}+\|\ou_2\|^2_{\V_{\text{div}}})\|\delta\ou\|^2_\sharp,$\\
      \item $|\mathcal{I}_4|\leq\frac{\nu_1}{12}\|\delta\ou\|^2+C\|\varphi_1\|^2_{\mathrm{H}^3}\|\delta\ou\|^2_\sharp,$\\
      \item $|\mathcal{I}_5|\leq \frac{\nu_1}{12}\|\delta\ou\|^2+C\|\varphi_1\|^4_{\mathrm{H}^2}\|\delta\ou\|^2_\sharp.$
  \end{itemize}
  The remaining terms in \eqref{4.8}, namely $\mathcal{I}_i$ for $i = 7, \ldots, 17$, require separate treatment. To estimate these, we make use of the a priori bounds derived in \eqref{estimate h}–\eqref{estimate ht}, the regularity properties of the weak solutions $(\u_i, \varphi_i)$ for $i = 1, 2$, and the $\mathrm{L}^4$ estimate for $\varphi_i$ provided in Remark~\ref{L4esti_phi}. The estimation process involves a combination of classical functional inequalities such as Hölder’s, Young’s, Poincaré’s, Gagliardo–Nirenberg, Agmon’s inequalities, and relevant Sobolev embeddings. Additionally, it is straightforward to verify that $\mathcal{I}_6 = 0$ due to its structure.

The estimates for the remaining terms are briefly presented below:
\begin{itemize}
\item $|\mathcal{I}_3|\leq C(\|\varphi_1\|^2_{\mathrm{H}^2}+\|\varphi_2\|^2_{\mathrm{H}^2})\|\delta\ou\|^2_\sharp+\frac{1}{8}\|\Delta\delta\varphi\|^2+C\|\delta\varphi\|^2,$\\
\item $|\mathcal{I}_7|\leq C\|\varphi_1\|^2_{\mathrm{H}^3}\|\delta\varphi\|^2+\frac{\nu_1}{12}\|\delta\ou\|^2 $,\\
\item $|\mathcal{I}_8|\leq C\|\varphi_1\|^2_{\mathrm{H}^2}\|\delta\varphi\|^2+C\|\delta\u_e\|^2_{\H^1_{\text{div}}}$, \\
    \item $|\mathcal{I}_9|\leq \frac{1}{8}\|\Delta(\delta\varphi)\|^2+C\|\u_{e_2}\|^\frac{2}{3}_{\H^1_{\text{div}}}\|\delta\varphi\|^2$,\\
 \item $|\mathcal{I}_{10}|\leq\frac{1}{2}\|\delta\ou\|^2_{\sharp}+\frac{1}{2}\|\delta\u_e\|^2_{\H^1_{\text{div}}}\|\nabla\ou_1\|^2$,\\
  \item $|\mathcal{I}_{11}|\leq\frac{\nu_1}{12}\|\delta\ou\|^2+C\|\u_{e_1}\|^2_{\H^1_{\text{div}}}\|\delta\ou\|^2_{\sharp}$,\\
    \item $|\mathcal{I}_{12}|\leq\frac{1}{2}\|\delta\ou\|^2_{\sharp}+\frac{1}{2}\|\u_{e_1}\|^2_{\H^1_{\text{div}}}\|\delta\u_{e}\|^2_{\H^1_{\text{div}}}$,\\
    \item $|\mathcal{I}_{13}|\leq\frac{\nu_1}{12}\|\delta\ou\|^2+C\|\u_{e_2}\|^2_{\H^1_{\text{div}}}\|\delta\ou\|^2_{\sharp},$\\
    \item $|\mathcal{I}_{14}|\leq\frac{1}{2}\|\nabla\ou_2\|^2\|\delta\ou\|^2_{\sharp}+\frac{1}{2}\|\delta\u_e\|^2_{\H^1_{\text{div}}},$\\
    \item $|\mathcal{I}_{15}|\leq \frac{1}{2}\|\delta\ou\|^2_{\sharp}+\frac{1}{2}\|\u_{e_2}\|^2_{\H^1_{\text{div}}}\|\delta\u_e\|^2_{\H^1_{\text{div}}},$\\
    \item $|\mathcal{I}_{16}|\leq\frac{1}{2}\|\delta\ou\|^2_{\sharp}+\frac{1}{2}\|\partial_t(\delta\u_e)\|^2$.
   \end{itemize}
    Since $\A^{-1}:\V'_{\text{div}}\to\V_{\text{div}}$, using Poincar\'e inequality we have $\|\A^{-1}(\delta\ou)\|\leq\|\nabla\A^{-1}(\delta\ou)\|$,  which is used in  the above calculations. Now, using (A1) and Lemma \ref{lem4.2} we estimate the first term of \eqref{4.8} as follows:
    \begin{align}
        \mathcal{I}_1= &\Big(\int_0^1\nu'\big(\theta\varphi_1+(1-\theta)\varphi_2\big) d\theta \, \delta\varphi \D\ou_2, \nabla\A^{-1}(\delta\ou)\Big)\no\\
        \leq& C\|\D\ou_2\|\|\delta\varphi\|_{\mathrm{H}^2}\|\delta\ou\|_\sharp\no\\
        \leq & C\|\D\ou_2\|\Big(\|\Delta(\delta\varphi)\|^2+|\langle\delta\varphi\rangle|^2\Big)^\frac{1}{2}\|\delta\ou\|_\sharp\no\\
        \leq & C\|\ou_2\|^2_{\V_{\text{div}}}\|\delta\ou\|^2_\sharp+\frac{1}{8}\|\Delta(\delta\varphi)\|^2+C|\langle\delta\varphi\rangle|^2.
    \end{align}
    Next using (A5) and Young's inequality we obtain the estimate for last term of \eqref{4.8}
    \begin{align}
        |\mathcal{I}_{17}|=&\Big(\int_0^1 F''\big(s\varphi_1+(1-s)\varphi_2\big) ds \, \delta\varphi, \Delta(\delta\varphi)\Big)\no\\ \leq &\frac{1}{8}\|\Delta(\delta\varphi)\|^2+C\big(\sum_{i=1}^2\|\varphi_i\|^{q-1}\|\varphi_i\|^{q-1}_{\mathrm{H}^2}\big)\|\delta\varphi\|^2. {\color{teal} \, q\leq 5}
    \end{align}
    Collecting all the estimates together and substituting in \eqref{4.8}, we are led to the differential inequality
   \begin{align}\label{fde}
    &\frac{1}{2}\frac{d}{dt} \big(\|\delta\ou\|^2_{\sharp}+\|\delta\varphi\|^2\big)+\frac{\nu_1}{2}\|\delta\ou\|^2+\frac{1}{2}\|\Delta(\delta\varphi)\|^2\leq C\Big(\|\varphi_1\|^2_{\mathrm{H}^2}+\|\varphi_2\|^2_{\mathrm{H}^2}+\|\varphi_1\|^2_{\mathrm{H}^3}\no\\&\qquad+\|\u_{e_2}\|^\frac{2}{3}_{\H^1_{\text{div}}}+\|\ou_1\|^2_{\V_{\text{div}}}+\|\ou_2\|^2_{\V_{\text{div}}}+\|\u_{e_1}\|^2_{\H^1_{\text{div}}}+\|\varphi_1\|^4_{\mathrm{H}^2}+1+\sum_{i=1}^2\|\varphi_i\|^{q-1}\|\varphi_i\|^{q-1}_{\mathrm{H}^2}\Big)\no\\&\qquad\quad\times\big(\|\delta\ou\|^2_{\sharp}+\|\delta\varphi\|^2\big)+C(\|\nabla\ou_1\|^2+ \|\u_{e_1}\|^2_{\H^1_{\text{div}}}+\|\u_{e_2}\|^2_{\H^1_{\text{div}}})\|\delta\u_e\|^2_{\H^1_{\text{div}}}\no\\&\qquad\quad+\frac{1}{2}\|\partial_t(\delta\u_e)\|^2+C|\langle\delta\varphi\rangle|^2.
   \end{align}
   Then, using the uniform Gronwall lemma [cf. \cite[Lemma 1.1]{temam1}] in \eqref{fde} and taking advantage of 
   \eqref{e18}-\eqref{estimate ht} and Remark \ref{L4esti_phi}, we obtain the desired estimate \eqref{main_fde}.
   \end{proof}
   As a consequence, we have the following uniqueness of weak solutions:
\begin{lemma}
Given initial data $(\u_0,\varphi_0)\in\V^0(\Omega)\times\mathrm{H}^1$ and boundary data $\h$ satisfying Assumption \ref{H}, the weak solution of \eqref{equ P} on $[0,T)$ is unique for any $T>0$.
\end{lemma}
\section{Strong Solution}
In this section, we establish the existence of a strong solution to \eqref{equ P}. The proof is carried out using the Faedo-Galerkin approximation method. To proceed, we first define the notion of a strong solution.
\begin{definition}\label{strong sol def}
Let $T$ be a positive constant.
A pair $(\u, \varphi)$ is said to be a strong solution of the system \eqref{equ P} if it is a weak solution, and, in addition, fulfills the following regularity:
\begin{equation}\label{stng_reg}
\left\{
\begin{aligned}
    \u & \in \L^\infty(0, T; \V^1(\Omega)) \cap \L^2(0, T; \V^2(\Omega)),\\
    \varphi & \in \mathrm{L}^\infty(0, T; \mathrm{H}^2) \cap \mathrm{L}^2(0, T; \mathrm{H}^4),\\
    \u_t&\in \L^2(0, T; \V^0(\Omega)),\\
    \varphi_t &\in \L^2(0, T; \L^2),\\
    \mu&\in\L^2(0,T; \mathrm{H}^2).\\
\end{aligned}
\right.
\end{equation}
\end{definition}
\begin{remark}
Since $\mathrm{H}^2(\Omega)$ can be continuously embedded in $C(\overline{\Omega}),$ any strong solution $\varphi$ is continuous on $[0, T)\times\overline{\Omega.}$
\end{remark}

\begin{theorem}\label{st}
Let  $(\u_0,\varphi_0 ) \in \V^1(\Omega)\times \mathrm{H}^2$.  Let $\nu$, $F$ satisfies Assumption \ref{prop of F}, and $\h$ satisfy the Assumption \ref{H}.
and the compatibility condition \eqref{compatibility}. Then, there exists a unique global strong solution of the system \eqref{equ P} in the sense of Definition \ref{strong sol def}.
\end{theorem}
\begin{proof}

As before let us define $\overline{\u}=\u -\u_e$, where $(\u,\varphi)$ is a weak solution of the system  \eqref{equ P}. and $\u_e $ is elliptic lifting. 
Note that by theorem \ref{elliptic lifting}, we have  the requisite regularity of $\u_e$ and hence it is sufficient to prove that $\overline{\u}$ and $\varphi$  have the regularity as in \ref{stng_reg}. 

Recall, $ (\overline{\u}, \varphi) $ is a weak solution of the system \eqref{7}. 
Let $\mathrm{P}:\mathbb{L}^2(\Omega) \rightarrow \G_{\text{div}}$ be the Helmholtz-Hodge orthogonal projection. We take the projection of equation $\eqref{7}_3$ to eliminate the pressure term, which can be recovered later using \cite[Proposition 1.1]{temam}. As we did for a weak solution, we are going to derive  appropriate a priori estimates using weak formulation of system of Galerkin approximations \eqref{dist for u}-\eqref{dist for mu} defined earlier. As before the aim is to get uniform estimates in the requisite spaces which will allow us to pass to the limit. The calculations given in the sequel are formal and can be justified with the help uniform estimates. In order to simplify the notation, we drop the superscript $n$ in $\overline{\u}^n, \varphi^n, \mu^n$.
Using more regular test functions namely $\A\overline{\u}$ as $\v $  and  $\Delta^2\varphi$  as $\psi$ in \eqref{dist for u} and \eqref{dist for phi} respectively, and using Remark \ref{nu} we get 
\begin{align}
    \frac{1}{2}\frac{d}{dt} \|\nabla\overline{\u}\|^2 &+  \nu \|\Delta\overline{\u}\|^2 \leq  -b(\overline{\u}, \overline{\u}, \A\overline{\u}) - b(\overline{\u}, \u_e, \A\overline{\u}) - b(\u_e, \overline{\u}, \A \overline{\u}) \no\\
    &\qquad- b(\u_e, \u_e, \A\overline{\u}) +(\mu \nabla \varphi, \A\overline{\u}) - \langle \partial_t \u_e, \A\overline{\u} \rangle = \sum_{i=1}^6 I_i.  \label{u strong form}
 \end{align}
 and
 \begin{align}
    \frac{1}{2}\frac{d}{dt} \|\Delta \varphi\|^2 = -(\overline{\u} \cdot \nabla \varphi, \Delta^2\varphi) - (\u_e \cdot \nabla \varphi, \Delta^2\varphi) + (\Delta \mu, \Delta^2\varphi) =  \sum_{i=1}^3 J_i. \label{phi strong form}
 \end{align}
By applying Gagliardo-Nirenberg inequality, Young's inequality, Poincar\'e inequality,  Agmon's inequality, Sobolev embeddings, and \eqref{prop2.1}, estimations for all the terms in the R.H.S. of equation \eqref{u strong form} are done. We have the following list of estimates:
\begin{itemize}
    \item $|I_1| \leq   \|\overline{\u}\|_{\mathbb{L}^4} \|\nabla\overline{\u}\|_{\mathbb{L}^4} \, \|\Delta\overline{\u}\| \leq C \|\overline{\u}\|_{\mathbb{L}^4}\|\overline{\u}\|^{\frac{3}{4}}  \|\Delta\overline{\u}\|^{\frac{5}{4}} 
     \leq C \|\overline{\u}\|^\frac{4}{3} \|\nabla\overline{\u}\|^2 + \frac{\nu}{12}\|\Delta\overline{\u}\|^2$,\\
   \item  $|I_2|   \leq \|\overline{\u}\|_{\mathbb{L}^4} \, \|\nabla \u_e\|_{\mathbb{L}^4} \|\Delta\overline{\u}\|  \leq C \|\nabla \overline{\u}\|^2 \|\u_e\|^2_{\V^2(\Omega)} + \frac{\nu}{12} \|\Delta\overline{\u}\|^2$,\\
 \item  $|I_3|   \leq \|\u_e\|_{\mathbb{L}^\infty} \|\nabla\overline{\u}\|\|\Delta \overline{\u}\|
    \leq C \|\u_e\|^2_{\V^2(\Omega)} \|\nabla\overline{\u}\|^2 + \frac{\nu}{12}\|\Delta \overline{\u}\|^2,$\\
   \item $|I_4|  \leq \|\u_e\|_{\mathbb{L}^4} \, \|\nabla\u_e\|_{\mathbb{L}^4} \, \|\Delta\overline{\u}\| \leq C \|\u_e\|^2_{\V^1(\Omega)} \|\u_e\|^2_{\V^2(\Omega)} + \frac{\nu}{12} \|\Delta\overline{\u}\|^2,$\\
     \item  $|I_5| \leq   C \|\mu \nabla \varphi\| \, \|\Delta\overline{\u}\|  \leq C \|\mu\|_{\mathrm{H}^1} \|\nabla\varphi\| \Big( \log e \frac{\|\nabla \varphi\|_{\mathrm{H}^1}}{\|\nabla \varphi\|} \Big) \, \|\Delta\overline{\u}\|  \leq C \|\mu\|^2_{\mathrm{H}^1} \|\nabla\varphi\|^2 + \frac{\nu}{12} \|\Delta\overline{\u}\|^2,$\\
     \item $|I_6|  \leq C\|\partial_t \u_e\|^2 + \frac{\nu}{12} \|\Delta\overline{\u}\|^2.$  
     \end{itemize}
     Substituting all the above estimates of $I_1$ to $I_6$ into the equation \eqref{u strong form}, we get the following expression:
    \begin{align}\label{e15}
\frac{1}{2}\frac{d}{dt}\|\nabla\overline{\u}\|^2 + \frac{\nu}{2} \|\Delta \overline{\u}\|^2 & \leq C \big( \|\u_e\|_{\V^1(\Omega)} ^2 \|\u_e\|^2_{\V^2(\Omega)} + \|\mu\|^2_{\mathrm{H}^1} \|\nabla\varphi\|^2 +\|\partial_t \u_e\| ^2  \big) \no \\ &  + C\big( \|\overline{\u}\|^\frac{4}{3} + \|\u_e\|^2_{\V^2(\Omega)} \big) \|\nabla\overline{\u}\|^2.
    \end{align}
    Utilizing a combination of  Ladyzhenskaya, Young's, Poincar\'e, Gagliardo-Nirenberg, Agmon's inequalities, and Sobolev embeddings again, we proceed to estimate all terms appearing on the R.H.S. of equation \eqref{phi strong form} as follows:
        \begin{align}\label{e10}
        &|J_1|  \leq \|\overline{\u}\|_{\mathbb{L}^4} \|\nabla\varphi\|_{\mathrm{L}^4} \|\Delta^2\varphi\| \leq C \|\nabla \overline{\u}\| \, \|\varphi\|_{\mathrm{H}^2} \, \|\Delta^2\varphi\| \no \\
        & \qquad \leq C \|\nabla \overline{\u}\|^2 \, \|\varphi\|^2_{\mathrm{H}^2} + \frac{1}{4} \|\Delta^2\varphi\|^2,\\
        &|J_2| \leq \|\u_e\|_{\mathbb{L}^{\infty}} \|\nabla\varphi\| \|\Delta^2\varphi\|  \leq C \|\u_e\|^2_{\V^2(\Omega)} \|\nabla\varphi\|^2 + \frac{1}{4}  \|\Delta^2\varphi\|^2, \\
        &J_3= (\Delta\mu, \Delta^2\varphi)  = (\Delta (- \Delta\varphi + F'(\varphi)), \Delta^2\varphi) \no \\
        & \qquad= - \|\Delta^2\varphi\|^2 + (\Delta F'(\varphi), \Delta^2\varphi) \no \\
        & \qquad= - \|\Delta^2\varphi\|^2 + (F'''(\varphi) (\nabla\varphi)^2, \Delta^2\varphi) + (F''(\varphi) \Delta\varphi, \Delta^2\varphi).\label{e4}
        \end{align}
 Combining  \eqref{e10}-\eqref{e4} into \eqref{phi strong form}, we arrive at following inequality:
  \begin{align}\label{6.9}
      \frac{1}{2}\frac{d}{dt} \|\Delta \varphi\|^2  + \frac{1}{2} \|\Delta^2\varphi\|^2 &\leq C \|\nabla \overline{\u}\|^2 \, \|\varphi\|^2_{\mathrm{H}^2}  + C \|\u_e\|^2_{\V^2(\Omega)} \|\nabla\varphi\|^2  \no \\
     & \quad +(F'''(\varphi) (\nabla\varphi)^2, \Delta^2\varphi) + (F''(\varphi) \Delta\varphi, \Delta^2\varphi).
  \end{align}
 Now, we estimate the last two terms on the R.H.S. of \eqref{6.9} using (A5), (A6), and Gagliardo-Nirenberg inequality:
    \begin{align}\label{e9}
    (F'''(\varphi) |\nabla\varphi|^2, \Delta^2\varphi) & \leq \|F'''(\varphi)\|_{\L^4} \||\nabla\varphi|^2\|_{\L^4}\|\Delta^2\varphi\| \no \\
     & \leq C_5(1+\|\varphi\|^{q-2}_{\L^{4q-8}})\|\nabla\varphi\|^2_{\L^8} \|\Delta^2\varphi\| \no \\
     & \leq C (1+\|\varphi\|^{q-2}_{\mathrm{H}^1})\|\varphi\|^2_{\mathrm{H}^2} \|\Delta^2\varphi\|\no  \\
     & \leq  \frac{1}{8} \|\Delta^2\varphi\|^2 +C (1+\|\varphi\|^{2q-4}_{\mathrm{H}^1})\|\varphi\|^4_{\mathrm{H}^2}.
     \end{align}
   Similarly, we have
   \begin{align}\label{e13}
       (F''(\varphi) \Delta\varphi, \Delta^2\varphi) &\leq \|F''(\varphi)\|_{\L^4}\|\Delta \varphi\|_{\L^4} \|\Delta^2\varphi\| \no \\
      &\leq  C(1+\|\varphi\|^{q-1}_{\L^{4q-4}}) \|\Delta \varphi\|_{\L^4} \|\Delta^2\varphi\| \no \\
      &\leq  \frac{1}{8} \|\Delta^2\varphi\|^2+C \| \varphi\|^2_{\mathrm{H}^3}(1+\|\varphi\|^{2q-2}_{\mathrm{H}^1}) 
   \end{align}
   Now, substituting \eqref{e9}-\eqref{e13} in \eqref{6.9}, we finally get
\begin{align}\label{e14}
     \frac{1}{2}\frac{d}{dt} \|\Delta \varphi\|^2 + \frac{1}{4} \|\Delta^2\varphi\|^2  \leq &C \|\nabla \overline{\u}\|^2 \, \|\varphi\|^2_{\mathrm{H}^2}  + C \|\u_e\|^2_{\V^2(\Omega)} \|\nabla\varphi\|^2 + C \| \varphi\|^2_{\mathrm{H}^3}(1+\|\varphi\|^{2q-2}_{\mathrm{H}^1}) \no \\
     & + C (1+\|\varphi\|^{2q-4}_{\mathrm{H}^1})\|\varphi\|^4_{\mathrm{H}^2}.
\end{align}
Adding \eqref{e15} and \eqref{e14}, we  have
\begin{align}\label{e16}
\frac{d}{dt} \Big( \|\nabla\overline{\u}\|^2 &+ \|\Delta \varphi\|^2\Big) +  \nu \|\Delta \overline{\u}\|^2 + \frac{1}{2} \|\Delta^2\varphi\|^2   \leq C\big( \|\overline{\u}\|^\frac{4}{3}  + \|\u_e\|^2_{\V^2(\Omega)} +\|\varphi\|^2_{\mathrm{H}^2}  \big)\no\\&\times (\|\nabla\overline{\u}\|^2+\|\Delta\varphi\|^2) +C\Big(
\|\u_e\|_{\V^1(\Omega)}^2 \|\u_e\|^2_{\V^2(\Omega)} + \|\mu\|^2_{\mathrm{H}^1} \|\nabla\varphi\|^2 +\|\partial_t\u_e\|^2  \no \\
&\quad+  \|\u_e\|^2_{\V^2(\Omega)} \|\nabla\varphi\|^2 +  \|\varphi\|^2_{\mathrm{H}^3}(1+\|\varphi\|^{2q-2}_{\mathrm{H}^1}) +(1+\|\varphi\|^{2q-4}_{\mathrm{H}^1})\|\varphi\|^4_{\mathrm{H}^2} \Big).
\end{align}
By integrating \eqref{e16} from $0$ to $t$,  we get
\begin{align}
    \|\nabla \overline{\u}(t)\|^2 &+ \|\Delta \varphi(t)\|^2 +  \nu\int_0^t \|\Delta \overline{\u}\|^2 + \frac{1}{2} \int_0^t \|\Delta^2\varphi\|^2 \leq \Big( \|\nabla\overline{\u}_0\|^2 + \|\Delta \varphi_0\|^2\Big) \no\\ 
&+ C \int_0^t \big(\|\overline{\u}\|^\frac{4}{3} + \|\u_e\|^2_{\V^2(\Omega)}  +\|\varphi\|^2_{\mathrm{H}^2} \big) (\|\nabla\overline{\u}\|^2+\|\Delta\varphi\|^2)   \no \\
&+C \int_0^t \big( \|\u_e\|_{\V^1(\Omega)} ^2 \|\u_e\|^2_{\V^2(\Omega)} + \|\mu\|^2_{\mathrm{H}^1} \|\nabla\varphi\|^2 +\|\partial_t\h\|^2_{\V^{\frac{1}{2}}(\Gamma)} + \|\u_e\|^2_{\V^2(\Omega)} \|\nabla\varphi\|^2 \no \\
& \qquad\qquad+   \|\varphi\|^2_{\mathrm{H}^3}(1+\|\varphi\|^{2q-2}_{\mathrm{H}^1}) +(1+\|\varphi\|^{2q-4}_{\mathrm{H}^1})\|\varphi\|^4_{\mathrm{H}^2}\big) .
\end{align}
Now, using Gronwall inequality  we obtain
\begin{align}\label{e17}
  \|\nabla\overline{\u}(t)\|^2 &+ \|\Delta \varphi(t)\|^2 +  \nu \int_0^t \|\Delta \overline{\u}\|^2 +  \frac{1}{2}\int_0^t \|\Delta^2\varphi\|^2 \leq \Big( \|\nabla\overline{\u}_0\|^2 + \|\Delta \varphi_0\|^2 \no\\   
  &+C \int_0^t \big(\|\u_e\|_{\V^1(\Omega)} ^2 \|\u_e\|^2_{\V^2(\Omega)} + \|\mu\|^2_{\mathrm{H}^1} \|\nabla\varphi\|^2 +\|\partial_t\h\|^2_{\V^{\frac{1}{2}}(\Gamma)}  \no \\
&\qquad\qquad+  \|\u_e\|^2_{\V^2(\Omega)} \|\nabla\varphi\|^2 + (1+\|\varphi\|^{2q-2}_{\mathrm{H}^1}) +(1+\|\varphi\|^{2q-4}_{\mathrm{H}^1})\|\varphi\|^4_{\mathrm{H}^2}\big)\Big) \no \\
&\times exp \int_0^t  \Big(\|\overline{\u}\|^\frac{4}{3}  + \|\u_e\|^2_{\V^2(\Omega)}  +\|\varphi\|^2_{\mathrm{H}^2} \Big),
\end{align}
for all $t \in [0, T).$
The R.H.S. of \eqref{e17} is finite due to Remark \ref{L4esti_phi}, $(\overline{\u},\varphi)$ satisfies \eqref{weak 2} and $\u_e$ satisfies \eqref{e18}. Thus we have,
\begin{align}\label{e18a}
   \|\nabla\overline{\u}(t)\|^2 &+ \|\Delta \varphi(t)\|^2 +  \nu\int_0^t \|\Delta \overline{\u}\|^2 +\int_0^t \|\Delta^2\varphi\|^2 \leq C , \ \ \forall t \in [0,T), 
\end{align}
which gives $\overline{\u} \in \mathrm{L}^\infty(0,T;\V_{\text{div}}) \cap \mathrm{L}^2(0,T;\H^2(\Omega))$ by using the elliptic regularity \cite[Theorem 3.1.2.1]{PG85}. Thus we have $\u  \in \mathrm{L}^\infty(0, T; \V^1(\Omega)) \cap \mathrm{L}^2(0, T; \V^2(\Omega))$. Furthermore, together with the no-flux boundary condition on $\varphi$ we get $\varphi  \in \mathrm{L}^\infty(0,T;\mathrm{H}^2).$ Combining Remark \ref{L4esti_phi}, estimate \eqref{e18a} and the no-flux boundary condition on $\mu$ we have  $\varphi  \in   \mathrm{L}^2(0, T; \mathrm{H}^4)$. 
  Also, we note that with this regularity of $\varphi$, Assumption \ref{prop of F} and using similar estimate as in \eqref{e9} and \eqref{e13}, we have 
   \begin{align}
       \int_0^t\|\Delta\mu\|^2\leq \int_0^t\|\Delta^2\varphi\|^2 + \int_0^t\|\Delta F'(\varphi)\|^2 <\infty.
   \end{align}
   Next, we provide a brief estimate for the time derivatives of $\overline{\u}$ and $\varphi$. From \eqref{ut projection} and \eqref{phit projection}, we obtain
     \begin{align}
         \int_0^t\|\partial_t\overline{\u}\|^2&\leq \nu\int_0^t\|\Delta\overline{\u}\|^2+\big(\|\overline{\u}\|^2_{\L^\infty(0, t;\V_{\text{div}})}+\|\u_e\|^2_{\mathrm{L}^\infty(0,T; \V^1(\Omega))}\big)\int_0^t\|\u_e\|^2_{\V^2(\Omega)}\no\\&\quad+\|\varphi\|^2_{\L^\infty(0, t; \mathrm{H}^2)}\int_0^t\|\mu\|^2_{\mathrm{H}^1}+\int_0^t\|\partial_t\h\|^2_{\V^{\frac{1}{2}}(\Gamma)} \quad\leq C\no.
     \end{align}
     and 
     \begin{align}
         \int_0^t\|\partial_t\varphi\|^2 &\leq (\|\u_e\|_{\L^\infty(0, t; \V^1)}+\|\overline{\u}\|_{\L^\infty(0, t; \V^1)})\int_0^t\|\nabla\varphi\|^2_{\mathbb{L}^4}+\int_0^t\|\Delta\mu\|^2
         \leq C.
     \end{align}
     Then, using similar reasoning as in Step 6, Step 7 of Theorem \ref{thm4.2}, we conclude that there exists a solution $(\overline{\u}, \varphi)$ satisfying \eqref{stng_reg}.
     Thus, we finally have  $\u  \in \mathrm{L}^\infty(0, T; \V^1(\Omega)) \cap \mathrm{L}^2(0, T; \V^2(\Omega))$ and $\varphi  \in \mathrm{L}^\infty(0, T;\mathrm{H}^2) \cap \mathrm{L}^2(0, T;\mathrm{H}^4)$.
     This completes the proof.
\end{proof}

\end{document}